\documentclass[a4paper]{amsart} 
\usepackage[pdftex]{graphicx,color} 
\usepackage{enumerate}
\usepackage{amssymb,amsmath,amscd}
\usepackage[mathscr]{euscript}

\usepackage{setspace} 

\newtheorem{theorem}{Theorem}[section] 
\newtheorem{lemma}[theorem]{Lemma} 
 
\newtheorem{proposition}[theorem]{Proposition} 
\newtheorem{corollary}[theorem]{Corollary} 

\theoremstyle{remark}
\newtheorem{remark}[theorem]{Remark} 

\theoremstyle{definition}
\newtheorem{definition}[theorem]{Definition}

\newcommand{\R} {{\mathbb R}} 
 
\newcommand{\Z} {{\mathbb Z}} 
\newcommand{\N} {{\mathbb N}}

\newcommand{\ep} {\epsilon}

\newcommand{\taz} {T_{z}} 
 
\newcommand{\tay} {T_{y}}

\newcommand{\D}{\partial}

\newcommand{\sman} {\mathcal M}

\newcommand{\innt} {\operatorname{int}}

\newcommand{\leb} {\operatorname{Leb}}

\newcommand{\wm} {\mathcal{M}}

\newcommand{\wa} {\mathcal{A}}  
\newcommand{\wb} {\mathcal{B}}
\newcommand{\we} {\mathcal{E}}  
\newcommand{\wu} {\mathcal{U}}  
\newcommand{\wv} {\mathcal{V}}
\newcommand{\ww} {\mathcal{W}}

\newcommand{\dx} {D_{x}}
\newcommand{\dy} {D_{y}}

\newcommand{\w}[1] {\tilde{#1}}

\renewcommand{\wr} {\mathcal{R}}
\newcommand{\wrn} {\mathcal{R}^{-}}
\newcommand{\wrp} {\mathcal{R}^{+}}

\newcommand{\wt} {\mathcal{F}}

\newcommand{\ws} {\mathcal{S}}
\newcommand{\wsn} {\mathcal{S}^{-}}
\newcommand{\wsp} {\mathcal{S}^{+}}
\newcommand{\wc} {\mathcal{C}}  
\newcommand{\wq} {\mathcal{Q}}         
\newcommand{\wk} {\mathcal{K}} 

\newcommand{\wy} {\mathscr{C}}
\newcommand{\Leb} {\mathcal{L}}

\renewcommand{\wp} {\mathcal{D}}
\newcommand{\wg} {\mathcal{G}_{\delta,c}}
\newcommand{\wh} {\mathcal{H}_{\delta,c,\alpha}}
\newcommand{\wj} {\mathcal{H}^{u}_{\delta,c,\alpha,M}}

\newcommand{\ch} {L_{\omega}}

\newcommand{\sic} {\sigma_{\wc}}
\newcommand{\sics} {\sigma^{*}_{\wc}}

\newcommand{\mau}{\mathcal{U}}

\graphicspath{{figures/}}

\begin{document} 
\title[A local ergodic theorem]{A local ergodic theorem for non-uniformly hyperbolic symplectic maps with singularities} 

\author[G. Del Magno]{Gianluigi Del Magno} \address{CEMAPRE, ISEG, Universidade Tecnica de Lisboa \\ 1200 Lisbon, Portugal} \email{delmagno@iseg.utl.pt} 
\author[R. Markarian]{Roberto Markarian} \address{Instituto de Matem\'atica y Estad\'istica ``Prof. Ing. Rafael Laguardia" (IMERL) \\
Facultad de Ingenier\'ia, Universidad de la Rep\'ublica \\
Montevideo, Uruguay} \email{roma@fing.edu.uy}

\date{\today}  

\thanks{The first author would like to thank M. Lenci and C. Liverani for helpful discussions. A preliminary version of this paper was started during a visit of the first author to IMERL (Uruguay), whereas the paper in the current version was completed during a visit of the second author to CEMAPRE (Portugal). Both authors acknowledge the warm hospitality of IMERL and CEMAPRE. The first author was partially supported by Funda\c{c}\~{a}o para a Ci\^{e}ncia e Tecnologia through the project `Randomness in Deterministic Dynamical Systems and Applications' (PTDC/MAT/105448/2008). The second author acknowledges `Proyecto 720', Universidad de la Rep\'ublica, Uruguay.}

\begin{abstract} 
In this paper, we prove a criterion for the local ergodicity of non-uniformly hyperbolic symplectic maps with singularities. Our result is an extension of a theorem of Liverani and Wojtkowski. 
\end{abstract}

\maketitle

In this paper, we consider a class of invertible maps with discontinuities and unbounded derivatives, which model billiards and other physical systems. Let $ \wt $ be one of these maps, and assume that $ \wt $ preserves a symplectic form and is non-uniformly hyperbolic. The last condition means that the Lyapunov exponents of $ \wt $ are non-zero with respect to the symplectic volume, which is invariant. Our main result (Theorem \ref{th:LET}) establishes sufficient conditions for a point in the domain of $ \wt $ to have a neighborhood contained up to a set of zero measure in one ergodic component of $ \wt $. In fact, we prove a stronger result, namely, that the mentioned neighborhood is contained up to a set of zero measure in one Bernoulli component of $ \wt $. Results of this type are often called `local ergodic theorems'. Local ergodic theorems (LET's for short) play an essential role in the proof of the ergodicity of non-uniformly hyperbolic systems. For instance, suppose that we know that the set $ X $ of all points to which our LET applies has full measure, and that the map $ \wt $ is topological transitive. Then, we can conclude quite easily that $ \wt $ is ergodic. 
From our LET, we derive another criterion for the ergodicity of $ \wt $ based on the topology of the set $ X $ (Corollary \ref{co:LET}). 

The LET presented here is an extension of the LET of Liverani and Wojtkowski \cite{liwo95}. The two theorems differ by one of their hypotheses, namely, the one assuming the existence of an invariant continuous cone field $ \wc $. Whereas Liverani and Wojtkowski assume that $ \wc $ is defined everywhere on the interior of the domain of $ \wt $ (see Condition C \cite[Section 7]{liwo95}), we assume that $ \wc $ is defined only on an open subset of the domain of $ \wt $. This paper originated from an attempt to use the LET of Liverani and Wojtkowski to prove that the non-uniformly hyperbolic billiards introduced by Donnay \cite{do91} and independently by Bunimovich \cite{bu92} are ergodic. However, these billiards in general admit only an invariant piecewise continuous cone field, and so the LET of Liverani and Wojtkowski does not apply to all of them, at least not in an obvious way. Instead, our LET applies to the generality of Donnay's and Bunimovich's billiards, and ultimately allows us to prove their ergodicity. A detailed proof of this claim will appear elsewhere. 

The ideas behind the proof of our LET can be traced back to several seminal works: the work of Hopf on the ergodicity of the geodesic flow on a surface of negative curvature \cite{h}, the work of Anosov on uniformly hyperbolic systems \cite{an}, and the work of Sinai on the ergodicity of dispersing billiards \cite{s}. Sinai's work is particularly important for us, because it outlines a general method for proving the ergodicity of non-uniformly hyperbolic systems with discontinuities and unbounded derivatives. The proof of the LET of Liverani and Wojtkowski, on which the proof of our LET builds, is an improvement of Sinai's method. Other improvements of Sinai's method were accomplished by Sinai and Bunimovich \cite{bs}, Sinai and Chernov \cite{sc}, Kr{\'a}mli, Sim{\'a}nyi and Sz{\'a}sz \cite{kss1}, Chernov \cite{c} and Markarian \cite{ma90}. We should also mention that Burns and Gerber \cite{bg} and Katok \cite{kb} obtained LET's for smooth maps and smooth flows preserving smooth volume and contact flows. 
		
The proof of our LET follows closely that of Liverani and Wojtkowski, but several additions to their proof are required. The main one is an improved version of the Tail Bound (Proposition \ref{pr:tb}). Our proof, like the one of \cite{liwo95}, relies on the Katok-Strelcyn theory \cite{ks}, which extends Pesin's theory \cite{pe} to maps with singularities (discontinuities and unbounded derivatives), and the technique of invariant cones developed by Wojtkowski to prove the existence of positive Lyapunov exponents \cite{wo85,liwo95}. Because of the difference in the assumption on the invariant cone field $ \wc $, the remaining hypotheses of our LET are slightly different than the corresponding hypotheses of \cite{liwo95}. 
		
The paper is organized as follows. In the first section, we describe the class of maps to which our LET applies. Several notions are required to formulate and prove the LET. We introduce the notions of monotone quadratic forms and invariant cone fields in the first section, and the notions of sufficient and essential points in the second section. In the second section, we also prove some preliminary results concerning the basic hyperbolic properties of the maps considered (Proposition \ref{pr:man1} and Lemma \ref{le:tangent}). 
The third section is devoted to the formulation of the LET (Theorem \ref{th:LET}). As a corollary of the LET, we obtain a criterion for the ergodicity of a map, based on the topology of the set of all sufficient points of the map (Corollary \ref{co:LET}). The fourth section contains the proof of the LET. 

\section{Basic definitions} 
\label{se:basic}
In this section, we present the class of dynamical systems for which the LET holds. We also introduce the concepts of monotone quadratic forms and invariant cone fields. 

\subsection{Symplectic maps with singularities} \label{su:symplecticmaps}

\begin{definition}
	\label{de:regular}
	A compact subset $ \wa $ of a $ C^{2} $ manifold $ \sman $ of dimension $ k \ge 2 $ is called \emph{regular} if it is a union of finitely many compact subsets $ \wa_{1},\ldots,\wa_{n} $ of $ (k-1) $-dimensional $ C^{2} $ submanifolds of $ \sman $ such that 
	\begin{enumerate}
	\item each $ \wa_{i} $ is equal to the closure of its interior, 
	\item $ \wa_{i} \cap \wa_{j} \subset \D \wa_{i} $ for $ i \neq j $, 
	\item the boundary of each $ \wa_{i} $ is a union of finitely many compact subsets of $ (k-2) $-dimensional submanifolds of $ \sman $. 
	\end{enumerate}
	The sets $ \wa_{1},\ldots,\wa_{n} $ are called \emph{(regular) components of} $ \wa $. 
\end{definition}

Let $ (\sman,\omega) $ be a $ C^{2} $ symplectic compact manifold, possibly with boundary and corners, of dimension $ 2d $. The boundary $ \D \sman $ is assumed to be regular. The symplectic form $ \omega $ is allowed to degenerate, but only on $ \D \wm $. The manifold $ \sman $ is also equipped with a Riemannian metric $ g $. The corresponding norm and distance are denoted by $ \|\cdot\| $ and $ d $, respectively. The volume measures on $ \sman $ generated by $ g $ and $ \omega $ are denoted by $ \Leb $ and $ \mu $, respectively. Since $ \Leb(\D \wm) = 0 $ (a direct consequence of Proposition \ref{pr:vol}) and $ \mu $ and $ \Leb $ are generated by volume forms on $ \wm \setminus \D \wm $, we easily see that $ \Leb $ and $ \mu $ are equivalent (i.e., mutually absolutely continuous) and $ \mu \le \tau \Leb $ for some constant $ \tau>0 $. Given a subset $ \wa \subset \sman $ and $ \ep>0 $, let $ \wa(\ep) = \{y \in \sman: d(y,\wa)<\ep \} $ be the $ \ep $-neighborhood of $ \wa $. 

\begin{remark}
Unless otherwise specified, `almost everywhere' statements in Sections \ref{se:basic}-\ref{se:statementlet} have to be understood with respect to the measure $ \Leb $. In fact, we could use $ \mu $ as the reference measure as well, because $ \Leb $ and $ \mu $ are equivalent.
\end{remark}

\begin{definition}
	\label{de:map}
	Let $ (\wm,\omega,g) $ be a symplectic manifold with a metric $ g $ as described above, and suppose that there exist two regular subsets $ \wsp_{1} $ and $ \wsn_{1} $ of $ \sman $ and a $ C^{2} $ diffeomorphism $ \wt:\sman \setminus (\D \sman \cup \wsp_{1}) \to \sman \setminus (\D \sman \cup \wsn_{1}) $ such that
	\begin{enumerate} 
		\item $ \ws^{\pm}_{1} \cap \D \sman \subset \D \ws^{\pm}_{1} $;
		\item $ \wt $ preserves $ \omega $; 
		\item the differential of $ \wt $ satisfies the following conditions:
		\begin{itemize}
			\item $ \log^{+} \|\dx \wt \| $, $ \log^{+} \|\dx \wt^{-1} \| \in L^{1}(\mu) $, where $ \log^{+} x=\max \{0,\log x\} $;
			\item there are two constants $ A>1 $ and $ b >0 $ such that \[ \| \dx^{2} \wt \| \le \frac{A}{ d(x,\wr)^{b}}, \]
			where $ \wr = \wr^{-}_{1} \cup \wr^{+}_{1} $ and $ \wr^{\pm}_{1} = \D \sman \cup \ws^{\pm}_{1} $. 
		\end{itemize}
	\end{enumerate}  
	The system $ (\wm,\omega,g,\wt) $ is called a \emph{symplectomorphism with singular set} $ \wr $. 
	
\end{definition}

\begin{remark}
	Condition (3) of Definition \ref{de:map} incorporates Conditions 1.2 and 1.3 of \cite{ks}. We also observe that in our setting, Condition 1.1 of \cite{ks} follows from the regularity of $ \wrn_{1} $ and $ \wrp_{1} $. For the convenience of the reader, we recall that Condition 1.1 of \cite{ks} reads as follows: there exist positive constants $ C $ and $ a $ such that $ \mu(\wr^{\pm}_{1}(\epsilon)) \le C \ep^{a} $ for every $ \epsilon>0 $ sufficiently small.
\end{remark}

Probably, the most notable example of a symplectomorphism with singularities is the map associated to a billiard system, i.e, the mechanical system consisting of a ball moving without frictions inside a bounded domain $ B $ of $ \R^{n} $ with elastic collisions at the boundary $ \D B $ (see \cite{cm}). 


\begin{remark}
	\label{re:notneeded}
	From Definition \ref{de:map}, it follows that the sets $ \sman \setminus \wrp_{1} $ and $ \sman \setminus \wrn_{1} $ have the same finite number of connected components, and the transformation $ \wt $ maps diffeomorphically each connected component of $ \sman \setminus \wrp_{1} $ into a connected component of $ \sman \setminus \wrn_{1} $. We do not require as in \cite{liwo95} that $ \wt(\wt^{-1}) $ have a homeomorphic extension up to the boundary of each connected component of $ \sman \setminus \wrp_{1}(\sman \setminus \wrn_{1}) $. The reason is that although this property is explicitly required in \cite{liwo95}, a careful analysis of the proofs in that paper reveals that it is in fact never used.  
	
	
\end{remark}

For every subset $ \mathcal{A} \subset \wm $, we adopt the convention that $ \wt \mathcal{A} = \wt (\mathcal{A} \setminus \wrp_{1}) $ and $ \wt^{-1} A = \wt^{-1} (\mathcal{A} \setminus \wrn_{1}) $. This way, we make sense of the expressions $ \wt \mathcal{A} $ and $ \wt^{-1} A $ although $ \wt $ and $ \wt^{-1} $ are, strictly speaking, defined only on the subsets of $ \wm \setminus \wrp_{1} $ and $ \wm \setminus \wrn_{1} $, respectively.

\begin{definition}
	\label{de:singular}
	For every $ n>1 $, define recursively the sets 
\begin{align*} 
\wrp_{n} & = \wrp_{n-1} \cup \wt^{-1} \wrp_{n-1}, \\
 \wrn_{n} & = \wrn_{n-1} \cup \wt \wrn_{n-1}.
\end{align*}
\end{definition}

\begin{remark}
The map $ \wt^{n}:\sman \setminus \wrp_{n} \to \sman \setminus \wrn_{n} $ is a $ C^{2} $ diffeomorphism preserving the form $ \omega $ for every $ n \ge 1 $. 
\end{remark}

Every regular set $ \wa \subset \sman $ is endowed with a natural measure $ \Leb_{\wa} $ defined as follows. If $ \wa_{1},\ldots,\wa_{n} $ are the components of $ \wa $, let $ \Leb_{\wa_{i}} $ be the restriction to $ \wa_{i} $ of the volume measure induced by $ g $ on the submanifold containing $ \wa_{i} $. The measure $ \Leb_{\wa} $ is then given by $ \Leb_{\wa}(\wb) = \bigcup^{n}_{i=1} \Leb_{\wa_{i}}(\wb) $ for every measurable $ \wb \subset \wa $. 

\begin{definition}
	Let $ \Leb_{+} = \Leb_{\wsp_{1}} $ and $ \Leb_{-} = \Leb_{\wsn_{1}} $.
\end{definition}

The following proposition (see \cite[Proposition 7.4]{liwo95}) will be used several times in the proof of the LET. 
                                                         
\begin{proposition}                
	\label{pr:vol}
	Let $ \wa $ be a regular subset of $ \sman $, and let $ \wb $ be a closed subset of $ \wa $. Then, 
	\[ \lim_{\ep \to 0^{+}} \frac{\Leb(\wb(\ep))}{\ep} = 2 \Leb_{\wa}(\wb). \]
\end{proposition}

We will also need the notion of characteristic line from symplectic geometry.

\begin{definition}
	Let $ (\mathcal{V},\omega) $ be a linear symplectic space, and let $ X $ a codimension 1 subspace of $ \mathcal{V} $. Then the \emph{characteristic line} $ \ch(X) $ is the skew-orthogonal complement of $ X $, i.e., 
	\[ 
	\ch(X) = \left\{u \in \mathcal{V}: \omega(u,v) = 0 \text{ for all } v \in X \right\}.
	\]
\end{definition}

\subsection{Quadratic forms and invariant cone fields} 
\label{su:quad}
We now summarized the relevant material on monotone quadratic forms and invariant cone fields from \cite{liwo95}. 

Let $ (\sman,\omega,g,\wt) $ be a symplectomorphism with singularities. Let $ U $ be an open subset of $ \sman $, and consider two families $ A=\{A_{x}\}_{x \in U} $ and $ B=\{B_{x}\}_{x \in U} $ of transverse Lagrangian subspaces $ A_{x},B_{x} \subset T_{x} \sman $ for $ x \in U $ such that the mappings $ x \mapsto A_{x} $ and $ x \mapsto B_{x} $ are measurable.                                 

\begin{definition}
	\label{de:quadratic}
We define a quadratic form $ \wq = \{\wq_{x}\}_{x \in U} $ on $ U $ associated to the transverse Lagrangian families $ A $ and $ B $ by
\[ \wq_{x}(u) = \omega_{x}(u_{1},u_{2}) \qquad \text{for every } u \in T_{x} \sman \text{ and } x \in U, \]
where $ u_{1} \in A_{x} $ and $ u_{2} \in B_{x} $ are uniquely defined by $ u=u_{1}+u_{2} $.
\begin{itemize}
	\item $ \wq $ is \emph{continuous} if the mappings $ x \mapsto A_{x} $ and $ x \mapsto B_{x} $ are continuous;
	\item $ \wq $ is \emph{monotone (with respect to $ \wt $)} if $ \wq_{\wt^{k} x}(D_{x} \wt^{k} u) \ge \wq_{x} (u) $ for every $ u \in T_{x} \sman $, $ x \in U $ and $ k>0 $ such that $ \wt^{k}x \in U $;
	\item $ \wq $ is \emph{eventually strictly monotone (with respect to $ \wt $)} if it is monotone, and for a.e. $ x \in U $, there exists $ k(x)>0 $ such that $ \wt^{k}x \in U $ and $ \wq_{\wt^{k(x)} x}(D_{x} \wt^{k(x)} u) > \wq_{x} (u) $ for every $ u \in T_{x} \sman \setminus \{0\} $. 
\end{itemize}
\end{definition}                                                        

\begin{definition}
	\label{de:cone}
Let $ \wq $ be the quadratic form associated to the transverse Lagrangian families $ A $ and $ B $. The cone field $ \wc = \{\wc(x)\}_{x \in U} $ on $ U $ associated to $ A $ and $ B $ is the family of closed cones given by 
\[ 
\wc(x) = \wq^{-1}_{x}\left([0,+\infty)\right) \subset T_{x} \sman \qquad \text{for every } x \in U,
\]
For every $ x \in U $, the interior of $ \wc(x) $ is the cone 
\[
\innt \wc(x) = \wq^{-1}_{x}\left((0,+\infty)\right) \cup \{0\} \subset T_{x} \sman. 
\] 
We say that
\begin{itemize} 
	\item $ \wc $ is \emph{continuous} if the mappings $ x \mapsto A_{x} $ and $ x \mapsto B_{x} $ are continuous;
\item $ \wc $ is \emph{invariant (with respect to $ \wt $)} if $ D_{x} \wt^{k} \wc(x) \subset \wc(\wt^{k} x) $ for every $ x \in U $ and $ k>0 $ such that $ \wt^{k}x \in U $;
\item $ \wc $ is \emph{eventually strictly invariant (with respect to $ \wt $)} if it is invariant, and for a.e. $ x \in U $, there exists an integer $ k(x)>0 $ such that $ \wt^{k(x)}x \in U $ and $ D_{x} \wt^{k(x)} \wc(x) \subset \innt \wc(\wt^{k(x)} x) $. 
\end{itemize}
In this paper, a cone field and the corresponding quadratic form are always associated to two families $ A $ and $ B $ of Lagrangian subspaces. However, to avoid cumbersome notation, every time we introduce a cone field, we will not specify the families $ A $ and $ B $.
\end{definition}

The relation between the monotonicity of $ \wq $ and the invariance of $ \wc $ is established by the following proposition.

\begin{proposition}[Theorem 4.4 of \cite{liwo95}]
The cone field $ \wc $ is invariant(eventually strictly invariant) if and only if $ \wq $ is monotone(eventually strictly monotone).  
\end{proposition} 

\begin{definition}
Let $ \wc $ be a cone field on an open set $ U $. We define the complementary cone field $ \wc' $ of $ \, \wc $ by replacing $ [0,+\infty) $ with  $ (-\infty,0] $ in the definition of $ \wc $. Analogously, we define the set $ \innt \wc'(x) $ by replacing $ (0,+\infty) $ with $ (-\infty,0) $ in the definition of $ \innt \wc(x) $. 
\end{definition}

The next lemma shows that the notion of invariance for $ \wc $ can be equivalently formulated using $ \wc' $. 

\begin{lemma}[Proposition 6.2 of \cite{liwo95}]
	Let $ \wc $ be a cone field on an open set $ U $. Then,
	\begin{itemize}
		\item $ \wc $ is invariant if and only if $ D_{x} \wt^{-1} \wc'(x) \subset \wc'(\wt^{-k} x) $ for every $ x \in U $ and $ k>0 $ such that $ \wt^{-k}x \in U $; 
		\item $ \wc $ is eventually strictly invariant if and only if it is invariant, and for a.e. $ x \in U $, there exists an integer $ k(x)>0 $ such that $ D_{x} \wt^{-k(x)} \wc'(x) \subset \innt \wc'(\wt^{-k(x)} x) $.
	\end{itemize}
\end{lemma}

In the next definition, we formalize the notion of least expansion of the iterates of $ D \wt $ with respect to the quadratic form $ \wq $, and with respect to $ \wq $ and the norm $ \|\cdot\| $ together.

\begin{definition}
Let $ \wc $ be an invariant cone field on $ U $, and let $ \wq $ be the quadratic form generating it. For every $ x \in U $ and $ k>0 $ such that $ \wt^{k}x \in U $, let 
\[
\sic(\dx \wt^{k}) = \inf_{u \in \innt \wc(x)} 
\sqrt{\frac{\wq_{\wt^{k} x}(\dx \wt^{k} u)}{\wq_{x}(u)}}
\] 
and
\[
\sics(\dx \wt^{k}) = \inf_{u \in \innt \wc(x)} 
\frac{\sqrt{\wq_{\wt^{k}x}(\dx \wt^{k} u)}}{\|u\|}.
\]
For $ k<0 $, we define $ \sic $ and $ \sics $ by replacing the cone field $ \wc $ in the definitions above with its complementary cone field $ \wc' $. 
\end{definition}

\begin{remark}
	\label{re:invariance}
From the invariance of $ \wc $, it follows immediately that $ \sic(\dx \wt^{k}) \ge 1 $. Furthermore, if $ \dx \wt^{k} \wc(x) \subset \innt \wc(\wt^{k}x) $, then $ \sic(\dx \wt^{k})>1 $ \cite[Proposition 6.1]{liwo95}. 
\end{remark}

To conclude this section, we introduce the notion of joint invariance for two cone fields. Roughly speaking, if the cone fields $ \wc_{1} $ and $ \wc_{2} $ are jointly invariant, then $ \wc_{2} $ can be though as an extension of $ \wc_{1} $, and vice versa.

\begin{definition}
	Let $ \wc_{1} $ and $ \wc_{2} $ be two cone fields defined on the open sets $ U_{1} $ and $ U_{2} $, respectively. We say that $ \wc_{1} $ and $ \wc_{2} $ are \emph{jointly invariant} if 
	\begin{itemize}
		\item $ D_{x}\wt^{k} \wc_{1}(x) \subset \wc_{2}(\wt^{k}x) $ for every $ x \in U_{1} $ and $ k>0 $ such that $ \wt^{k}x \in U_{2} $,
		\item $ D_{x}\wt^{k} \wc_{2}(x) \subset \wc_{1}(\wt^{k}x) $ for every $ x \in U_{2} $ and $ k>0 $ such that $ \wt^{k}x \in U_{1} $.
	\end{itemize}
\end{definition}

\begin{remark}
	Note that in the previous definition, we neither require that the sets $ U_{1} $ and $ U_{2} $ are disjoint nor that the cone fields $ \wc_{1} $ and $ \wc_{2} $ are invariant. However, it is easy to see that $ \wc_{1} $ and $ \wc_{2} $ are invariant in the following sense: if $ x \in U_{1} $ and $ k_{2}>k_{1}>0 $ such that $ \wt^{k_{1}}x \in U_{2} $ and $ \wt^{k_{2}}x \in U_{1} $, then $ \dx \wt^{k_{2}} \wc_{1}(x) \subset \wc_{1}(\wt^{k_{2}}x) $. The same is true for $ \wc_{2} $, once $ U_{1} $ has been replaced by $ U_{2} $.
\end{remark}

\section{Sufficient points, essential points and non-uniform hyperbolicity}
\label{su:sepoints}


The notions of sufficient and essential points introduced in this section are borrowed from \cite{c}. 

\begin{definition}
	\label{de:suff} A point $ x \in \sman \setminus \D \sman $ is called \emph{sufficient} if there exist 
	\begin{enumerate}[(i)] 
		\item an integer $ l $ such that $ \wt^{l} $ is a local diffeomorphism at $ x $, 
		\item a neighborhood $ U $ of $ \wt^{l}x $ and an integer $ N>0 $ such that $ U \cap \wrn_{N} = \emptyset $, 
		\item an invariant continuous cone field $ \wc $ on $ U \cup \wt^{-N} U $ such that $ \sic(D_{y} \wt^{N}) > 3 $ for every $ y \in \wt^{-N} U $.
	\end{enumerate} 
To emphasize the role of $ l,N,U $ and $ \wc $ in this definition, we say that $ x $ is a sufficient point with quadruple $ (l,N,U,\wc) $. 
\end{definition}

\begin{remark}
The specific amount of expansion $ \sic(D_{y} \wt^{N})>3 $ in the definition of a sufficient point is required only for maps with singularities. For smooth maps, the weaker condition $ \sic(D_{y} \wt^{N})>1 $ suffices. The condition $ \sic(D_{y} \wt^{N})>3 $ is used in the proof of Proposition \ref{pr:sinai1} (see also Proposition 12.2 of \cite{liwo95}). 
\end{remark}

\begin{remark}
	\label{re:everysufficient}
Every point of the neighborhood $ U $ is a sufficient point with quadruple $ (0,N,U,\wc) $. 
\end{remark}

The cone field $ \wc $ in the definition of a sufficient point is clearly eventually strictly invariant. By a well-known result of Wojtkowski \cite{wo85} (see also \cite{ma88,kb}, for the same result formulated using of quadratic forms), it follows that all the Lyapunov exponents of $ \wt $ are non-zero a.e. on the set $ \bigcup_{k \in \Z} \wt^{k}U $. This fact and the Katok-Strelcyn theory \cite{ks} imply Proposition \ref{pr:man1} below. Actually, Claim (3) of the proposition follows from Proposition \ref{pr:ref2} in Section \ref{se:proof}. 

\begin{proposition} 
	\label{pr:man1}
	Let $ x \in \sman \setminus \D \sman $ be a sufficient point with quadruple $ (l,N,\wc,U) $. Then, there exist an invariant measurable set $ \Lambda \subset \bigcup_{k \in \Z} \wt^{k}U $ with $ \mu(\bigcup_{k \in \Z} \wt^{k}U \setminus \Lambda) = 0 $ and two families of $ C^{2} $ submanifolds  $ V^{s} = \{V^{s}_{y}\}_{y \in \Lambda} $ and $ V^{u} = \{V^{u}_{y}\}_{y \in \Lambda} $ such that for every $ y \in \Lambda $, the following hold
	\begin{enumerate}
		\item $ V^{s}_{y} \cap V^{u}_{y} = \{y\} $,
		\item $ V^{s}_{y} $ and $ V^{u}_{y} $ are embedded $ d $-dimensional balls,  
		\item $ \tay V^{s}_{y} \subset \wc'(y) $ and $ \tay V^{u}_{y} \subset \wc(y) $  provided that $ y \in U \cup \wt^{-N}U $,
		\item $ \wt V^{s}_{y} \subset V^{s}_{\wt y} $ and $ \wt^{-1} V^{u}_{y} \subset V^{u}_{\wt^{-1}y} $,
		\item $ d(\wt^{n}y,\wt^{n}z) \to 0 $ exponentially as $ n \to +\infty $ for every $ z \in V^{s}_{y} $, and the same is true as $ n \to -\infty $ for every $ z \in V^{u}_{y} $.
	\end{enumerate}
		Furthermore, $ V^{s}_{y} $ and $ V^{u}_{y} $ vary measurably with $ y \in \Lambda $, and the families $ V^{s} $ and $ V^{u} $ have the absolute continuity property. 
\end{proposition} 

For the definition of the absolute continuity property of a family of submanifolds, we refer the reader to the books \cite{cm,ks}.
		
\begin{definition}		
		The submanifolds forming the families $ V^{s} $ and $ V^{u} $ are called \emph{local stable manifolds} and \emph{local unstable manifolds}, respectively.
\end{definition}

	Let $ x $ be a sufficient point of $ \sman \setminus \D \sman $, and let $ \Lambda $ be the set as in Proposition \ref{pr:man1}. For every $ y \in \Lambda $, denote by $ W^{u}_{y} $ the connected component of $ \bigcup_{k \ge 0} \wt^{k} V^{u}_{\wt^{-k}y} $ containing $ y $. Analogously, denote by $ W^{s}_{y} $ the set obtained by replacing $ \wt $ with $ \wt^{-1} $ and $ V^{u} $ with $ V^{s} $ in the definition of $ W^{u}_{y} $.
	The sets $ W^{s}_{y} $ and $ W^{u}_{y} $ are immersed submanifolds of $ \wm $. 

In the next lemma, we show that Claim (3) of Proposition \ref{pr:man1} remains valid for cone fields jointly invariant with $ \wc $. The lemma is a slight generalization of Lemma 5.3 of \cite{kb}.

\begin{lemma}
	\label{le:tangent}
	Let $ x $ be a sufficient point of $ \wm $ with quadruple $ (l,N,U,\wc) $, and let $ \Lambda $ be the set as in Proposition \ref{pr:man1}. Assume that $ \wp $ is a cone field on an open set $ V $ such that $ \wc $ and $ \wp $ are jointly invariant. Then, for a.e. $ y \in \Lambda $, we have 
	\begin{align*}	
		T_{z} W^{u}_{y} \subset \wp(z) & \qquad \text{for } z \in W^{u}_{y} \cap V, \\
		T_{z} W^{s}_{y} \subset \wp'(z) & \qquad \text{for } z \in W^{s}_{y} \cap V.
	\end{align*}
\end{lemma}

\begin{proof}
	We will prove only the unstable part of the lemma, because the stable one can be proved in the same way. 
	
	Luzin's Theorem implies that for every $ \ep>0 $, there is a closed set $ \Lambda_{\ep} \subset \Lambda $ with $ \mu(\Lambda \setminus \Lambda_{\ep})<\ep $ such that the maps $ w \mapsto V^{s}_{w} $ and $ w \mapsto V^{u}_{w} $ are uniformly continuous in the $ C^{1} $-topology on $ \Lambda_{\ep} $ (see \cite[Statement 7.1.3 of Theorem 7.1, Part I]{ks}). It is clear that $ \bigcup_{\ep>0} \Lambda_{\ep} = \Lambda $ (mod 0). For every $ \ep>0 $, we define a measurable set $ U_{\ep} \subset U $ as follows: if $ \mu(\Lambda_{\ep} \cap U) = 0 $, then set $ U_{\ep} = \emptyset $, otherwise, we use the Poincar\'e Recurrence Theorem to find a measurable set $ U_{\ep} = \Lambda_{\ep} \cap U $ (mod 0) such that for every $ w \in U_{\ep} $, there are two monotone sequences $ m_{i} \to +\infty $ and $ n_{i} \to -\infty $ as $ i \to +\infty $ for which $ \wt^{n_{i}(m_{i})}w \in U_{\ep} $ and $ d(\wt^{n_{i}(m_{i})}w,w)<1/i $ for all $ i>0 $. Since $ \bigcup_{\ep>0} U_{\ep} = U $ (mod 0), it follows that $ \bigcup_{k \in \Z} \bigcup_{\ep>0}  \wt^{k} U_{\ep} = \bigcup_{k \in \Z} \wt^{k} U = \Lambda $ (mod 0). 
	
	Now, let $ y \in \bigcup_{k \in \Z} \bigcup_{\ep>0} \wt^{k} U_{\ep} $. It follows that there exist $ \bar{\ep}>0 $, $ k>0 $ and $ w \in U_{\bar{\ep}} \neq \emptyset $ such that $ \wt^{k}w = y $. Let $ \{n_{i}\}_{i \in \N} $ be a sequence for $ w $ as described above. The continuity of $ V^{u} $ on $ \Lambda_{\bar{\ep}} $, the continuity of $ \wc $ on the open set $ U $ and Statement (3) of Proposition \ref{pr:man1} easily imply that there exists $ \delta>0 $ such that if $ w' \in U_{\bar{\ep}} $ with $ d(w',w)<\delta $ and $ w'' \in V^{u}_{w'} $ with $ d(w'',w)< 2\delta $, then $ T_{w''} V^{u}_{w'} \subset \wc(w'') $. Let $ z \in W^{u}_{y} \cap V $. From the definition of $ W^{u} $ and the properties of the sequence $ \{n_{i}\}_{i \in \N} $, we can find $ i>0 $ such that $ d(\wt^{n_{i}-k}y,\wt^{n_{i}-k}z) < \delta $ and $ d(\wt^{n_{i}}w,w) < \delta $. By previous observation, it follows that $ T_{\wt^{n_{i}-k}z} \wt^{n_{i}-k} W^{u}_{y} \subset \wc(\wt^{n_{i}-k}z) $. Since $ \wc $ and $ \wp $ are jointly invariant, we finally obtain 
		\begin{align*}
		\taz W^{u}_{y} & = D_{\wt^{n_{i}-k} z} \wt^{-n_{i}+k} \left(T_{\wt^{n_{i}-k} z} W^{u}_{\wt^{n_{i}-k} y}\right) \\
		& \subset D_{\wt^{n_{i}-k} z} \wt^{-n_{i}+k} \wc(\wt^{n_{i}-k} z) \subset \wp(z).
		\end{align*}	
\end{proof}

We now give the definition of an \emph{essential} point. Essential points appear in the formulation of Condition L3 of our LET (Theorem \ref{th:LET}). These points play the same role as the points with strictly unbounded derivatives in the Sinai-Chernov Ansatz of Liverani and Wojtkowski (see Condition F in Section 7 of \cite{liwo95}).

\begin{definition}
	\label{de:ess} A point $ x \in \sman \setminus \D \sman $ is called \emph{u-essential} if for every $ \alpha>0 $, there exist
	\begin{enumerate}[(i)]
		\item a neighborhood $ U $ of $ x $ and an integer $ n>0 $ such that $ U \cap \wrp_{n} = \emptyset $, 
		\item an invariant continuous cone field $ \wc $ on $ U \cup \wt^{n} U $ such that $ \sics(D_{y} \wt^{n})>\alpha $ for every $ y \in U $.
	\end{enumerate}
	Analogously, we define an \emph{s-essential point} by replacing in the definition above $ \wt $ and $ \wrp_{n} $ with $ \wt^{-1} $ and $ \wrn_{n} $, respectively. 
\end{definition}

\begin{definition}
	\label{de:familyofcones}
From the definition of an essential point, we obtain a whole family of invariant continuous cone fields -- in fact, one for each value of $ \alpha $ -- associated to an essential point $ y \in \sman $. We call such a family the \emph{family of invariant cone fields associated to} $ y $. 
\end{definition}

We recall that if $ E $ is an ergodic component of positive measure and with non-zero Lyapunov exponents, then by Theorem 13.1 of \cite[Part II]{ks}, there exist $ m>0 $ disjoint measurable sets $ B_{1},\ldots,B_{m}=B_{0} $ of $ \sman $ such that 
\begin{enumerate}
	\item $ E = \bigcup^{m-1}_{i=0} B_{i} $, 
	\item $ \wt B_{i} = B_{i+1} $ for each $ i=0,\ldots,m-1 $,
	\item the restriction $ \wt^{m}|_{B_{i}} $ is a K-automorphism for each $ i=0,\ldots,m-1 $. 
\end{enumerate}
	The map $ \wt^{m}|_{B_{i}} $ is in fact Bernoulli \cite{ch,ow}. The sets $ B_{1},\ldots,B_{m} $ are uniquely defined up to a set of zero measure, and are called \emph{Bernoulli components} of $ \wt $.

\section{Local Ergodic Theorem}
\label{se:statementlet}

We now formulate our LET, but postpone its proof to Section \ref{se:proof}. In this section, we also formulate and prove Corollary \ref{co:LET}, which represents a useful criterion for the ergodicity of the map $ \wt $. We end the section by commenting on the hypotheses of the LET.

\begin{theorem}[LET]   
	\label{th:LET}
	Let $ x \in \sman \setminus \D \sman $ be a sufficient point with quadruple $ (l,N,U,\wc) $. Furthermore, let $ \Lambda $ be the subset of $ \, \bigcup_{k \in \Z} \wt^{k} U $ as in Proposition \ref{pr:man1}, and suppose that Conditions L1-L4 below are satisfied. 
\begin{description}
\item[L1 (Regularity)] The sets $ \wrp_{k} $ and $ \wrn_{k} $ are regular for every $ k > 0 $. 
\item[L2 (Alignment)] For every $ k>0 $, we have
	\begin{itemize}
	\item if $ \Sigma $ is a component of $ \wrn_{k} $ and $ y \in \Sigma \cap \wt^{-N} U $, then
	\[ 
	\ch(\tay \Sigma) \subset \wc(y),
	\]
	\item if $ \Sigma $ is a component of $ \wrp_{k} $ and $ y \in \Sigma \cap U $, then
	\[ 
	\ch(\tay \Sigma) \subset \wc'(y).
	\]
	\end{itemize}
\item[L3 (Sinai-Chernov Ansatz)]
	The set of all u(s)-essential points of $ \wsn_{1}(\wsp_{1}) $ has full $ \Leb_{-}(\Leb_{+}) $-measure, and if $ y $ is one of such points, then $ \wc $ and each cone field of the family of the invariant cone fields associated to $ y $ (see Definition \ref{de:familyofcones}) are jointly invariant. 
\item[L4 (Contraction)]
	There exist $ \beta>0 $ and $ \xi>0 $ such that 
	\begin{itemize}
		\item if $ y \in \Lambda \cap U $, $ z \in W^{u}_{y} $ and $ \wt^{-k}z \in \wsn_{1}(\xi) $ with $ k>0 $, then
	\[
	\left\|D_{z} \wt^{-k}|_{T_{z} W^{u}_{y}} \right\| \le \beta, 
	\]
	\item if $ y \in \Lambda \cap U $, $ z \in W^{s}_{y} $ and $ \wt^{k}z \in \wsp_{1}(\xi) $ with $ k>0 $, then
	\[
	\left\|D_{z} \wt^{k}|_{T_{z} W^{s}_{y}} \right\| \le \beta. 
	\]
	\end{itemize}
\end{description}
Then, there exists a neighborhood $ \mathcal{O} $ of $ x $ contained (mod 0) in a Bernoulli ergodic component of $ \wt $. 
\end{theorem}

\begin{remark}
	We have already observed that if the point $ x $ is sufficient, then every point of the neighborhood $ U $ is also sufficient (see Remark \ref{re:everysufficient}). Since Conditions L1-L4 depend only on $ N,U,\wc $ and not on $ x $, we see that if the LET applies to $ x $, then it does to every point of $ U $.
\end{remark}

The next corollary is a straightforward consequence of the LET. It is the analog for symplectomorphisms with singularities of Theorem 4.1 of \cite{kb}.

\begin{corollary}
	\label{co:LET}
	Under the same hypotheses of Theorem \ref{th:LET}, we have
	\begin{enumerate}
		\item each Bernoulli component of $ \wt $ contained in $ \bigcup_{k \in \Z} \wt^{k}U $ is open (mod 0);
		\item each connected component of $ \, \bigcup_{k \in \Z} \wt^{k}U $ is contained (mod 0) in a Bernoulli component of $ \wt $. In particular, if $ \, \bigcup_{k \in \Z} \wt^{k}U $ is connected, then $ \, \bigcup_{k \in \Z} \wt^{k}U $ coincides (mod 0) with a Bernoulli component of $ \wt $. 
	\end{enumerate}
\end{corollary}

\begin{proof}
	Let $ B $ be a Bernoulli component of $ \wt $ contained in the set $ \bigcup_{k \in \Z} \wt^{k}U $. If $ y \in B $, then there is an integer $ k $ such that $ \wt^{k}y \in U $. Thus, every $ y \in B $ is sufficient, and satisfies Conditions L1-L4. By using the LET, we can then conclude that every $ y \in B $ has a neighborhood $ \mathcal{O}_{y} $ contained (mod 0) in $ B $. It follows that $ B \subset \bigcup_{y \in B} \mathcal{O}_{y} $, and since $ \sman $ is a Lindel\"of space, we can extract a countable subcover of $ B $ from $ \{\mathcal{O}_{y}\}_{y \in B} $. Thus, we see that $ B $ is open (mod 0), and Part (1) of the corollary is proved.
	
	Let $ C $ be a connected component of $ \, \bigcup_{k \in \Z} \wt^{k}U $. Similarly as above, we can show that every point $ y \in C $ has a neighborhood $ \mathcal{O}_{y} $ contained (mod 0) in a Bernoulli component of $ \wt $. It is straightforward to see that the set 
	\[ \mathcal{O}_{C}=\bigcup_{y \in C} \mathcal{O}_{y} \]
	 is connected. Therefore, $ \mathcal{O}_{C} $ has to be contained (mod 0) in a single Bernoulli component of $ \wt $. The same is true for $ C $, because $ C \subset \mathcal{O}_{C} $. This proves Part (2) of the corollary, and completes the proof. 
\end{proof}

Some comments on the hypotheses of the LET are in orders. In the following, we will mainly try to elucidate the difference between the hypotheses of our LET and those of the LET of Liverani and Wojtkowski.

\subsubsection*{{\bf Cone field:}}  
The invariant cone field $ \wc $ in our definition of a sufficient point is defined on $ U \cup \wt^{-N} U $ with $ U $ being an open subset of $ \sman \setminus \D \sman $. Instead, Liverani and Wojtkowski assume that $ \wc $ is defined on the entire set $ \sman \setminus \D \sman $. We show in the next section, that our weaker condition on $ \wc $ suffices to prove the LET. The drawback is that the other hypotheses of our LET turn out to be more involved than those of the LET of \cite{liwo95}.


		
\subsubsection*{{\bf L1:}} Condition L1 is identical to the condition called Regularity in \cite[Section 7]{liwo95}).
		
\subsubsection*{{\bf L2 and L4:}} 
Condition L2 corresponds to the condition called Proper Alignment in \cite[Section 7]{liwo95}. It is not difficult to see that the Proper Alignment implies L2. Liverani and Wojtkowski use L2 in the proof of their LET, and not the full Proper Alignment (see the proof of Proposition 12.2 of \cite{liwo95}). A similar remark can be made for our L4 and the Noncontraction condition of \cite[Section 7]{liwo95}. There is an important difference between L2 and the Proper Alignment. The Proper Alignment requires the characteristic lines of the tangent spaces of the singular sets to be contained in the interior of the cones, whereas in L2 these characteristic lines have just to be in the cones. The strict inclusion assumed by the Proper Alignment is used in the original proof of the Tail Bound \cite[Section 13]{liwo95}, but not in our proof (see Proposition \ref{pr:tb}).

		
\subsubsection*{{\bf L3:}} 

The difference between Condition L3 and the Sinai-Chernov Ansatz of \cite[Section 7]{liwo95} is due to the fact that we do not assume the cone field $ \wc $ to be defined on the singular sets $ \wsp_{1} $ and $ \wsn_{1} $.
It is precisely to remedy to this situation that we have to introduce the concepts of essential points and joint invariant cone fields in the formulation of L3.

\section{Proof of the local ergodic theorem}  
\label{se:proof}

In this section, we prove the LET. The proof presented here follows closely the one of the Main Theorem (Discontinuous case) of \cite[Section 7]{liwo95}. In fact, except for the Tail Bound \cite[Section 13]{liwo95}, a large portion of the proof of Liverani and Wojtkowski retains its validity under Conditions L1-L4 without significant modification. Since the proof of the LET is quite long, we split it into several parts. Each part forms a subsection of the current section. We do not repeat the proofs of \cite{liwo95} which remain valid in our setting. Instead, we refer the reader to the original results, and limit ourselves to explain why these results extend to our setting. The reader should be cautioned that our notation does not always match the one of \cite{liwo95}. As already mentioned, the part of proof of Liverani and Wojtkowski that does not hold in our setting is the so-called Tail Bound. Our new proof of the Tail Bound (Proposition \ref{pr:tb}) is laid out in Subsection \ref{su:tb}.

\subsection{Reference neighborhood}  
\label{su:refn}

Let $ x \in \sman \setminus \D \sman $ be a sufficient point with quadruple $ (l,N,U,\wc) $. The first step of the proof of Theorem \ref{th:LET} consists in constructing a neighborhood $ \wu_{\rho} $ of the point $ \wt^{l}x $ endowed with a cone field $ \wy_{\rho} $ such that after every sufficiently long return to $ \wu_{\rho} $, vectors contained in $ \wy_{\rho} $ are expanded uniformly. The precise results are Propositions \ref{pr:ref1} and \ref{pr:ref2} below. The proof of the first proposition can be derived in a straightforward manner from the considerations at pages 41 and 42 of Section 8 of \cite{liwo95}, whereas the proof of the second proposition is exactly as that of Proposition 8.4 of \cite{liwo95}. These proofs are still valid in our setting, because they rely only on the sufficiency of the point $ x $. We observe that the definition of sufficiency is not explicitly given in \cite{liwo95}: it is just assumed that there exists an eventually strictly invariant continuous cone field $ \wc $ on $ \sman \setminus \D \sman $ and $ \sic(\dx \wt^{N})>3 $ for some $ x \in \sman \setminus \D \sman $.
                                                                
Recall that the dimension of the manifold $ \sman $ is $ 2d $. Denote by $ | \cdot | $ and $ \omega_{0} $ the Euclidean norm on $ \R^{d} $ and the standard symplectic form on $ \R^{d} \times \R^{d} $, respectively. For every $ a>0 $, let $ \wv_{a} $ be the $ d $-dimensional cube of size $ a $, i.e.,
\[ 
\wv_{a} = \left\{y=(y^{1},\ldots,y^{d}) \in \R^{d}:|y^{i}| < a \text{ for } i=1,\ldots,d \right\},
\]
and let
\[ 
\ww_{a} = \wv_{a} \times \wv_{a}.
\]

For each $ i=1,2 $, let $ \pi_{i}:\R^{d} \times \R^{d} \to \R^{d} $ be the projection given by $ \pi_{i}(u) = u_{i} $ for every $ u = (u_{1},u_{2}) \in \R^{d} \times \R^{d} $. Given a set $ A \subset \R^{d} \times \R^{d} $, we refer informally to the sets $ \pi_{1}(A) $ and $ \pi_{2}(A) $ as the projections of $ A $ onto the first and the second component, respectively, of $ \R^{d} \times \R^{d} $.

For every $ \rho>0 $, let the cone $ \wy_{\rho} \subset \R^{d} \times \R^{d} $ together with its complementary cone $ \wy'_{\rho} $ be given by
\begin{gather*}
\wy_{\rho} = \{u \in \R^{d} \times \R^{d}: |\pi_{2}(u)| \le \rho |\pi_{1}(u)| \}, \\
\wy^{'}_{\rho} = \{u \in \R^{d} \times \R^{d}: |\pi_{2}(u)| \ge \rho |\pi_{1}(u)| \}.
\end{gather*}

Let $ x $ be a sufficient point of $ \wm $ with quadruple $ (l,N,U,\wc) $, and let 
\[ 
\w{\rho} = \frac{1}{\sic(D_{\wt^{l-N}x} \wt^{N})}.
\] 
By definition of sufficiency, it follows that $ 0 < \w{\rho} < 1/3 $.
 
\begin{proposition}
	\label{pr:ref1} 
	For every $ \w{\rho}<\rho<1 $, there exist $ a_{\rho}>0 $ and a chart $ (\wu_{\rho},\Phi_{\rho}) $ with $ \wt^{l}x \in \wu_{\rho} \subset U $ such that
	\begin{enumerate}
		\item $ \Phi_{\rho}: \wu_{\rho} \to \ww_{a_{\rho}} $ is a diffeomorphism and $ \, \Phi^{*}_{\rho} \, \omega_{0} = \omega $, 
		\item $ \dy \Phi_{\rho} (\wc'(y)) \subset \wy^{'}_{1/\rho} $ for every $ y \in \wu_{\rho} $,
		\item $ \dy \Phi_{\rho} (\wt^{N} \wc(y)) \subset \wy_{\rho} $ for every $ y \in \wt^{-N} \wu_{\rho} $.
	\end{enumerate}
\end{proposition}

Note that Claim (1) of the previous proposition is just Darboux's Theorem \cite{ar97}. Let $ g_{\rho} = (\Phi^{-1}_{\rho})^{*}g $ be the Riemannian metric on $ \ww_{a_{\rho}} $ induced by $ g $ via $ \Phi^{-1}_{\rho} $. The metric $ g_{\rho} = (\Phi^{-1}_{\rho})^{*} g $ and the Euclidean metric of $ \R^{d} \times \R^{d} $ are clearly equivalent on $ \ww_{a_{\rho}} $. In view of Proposition \ref{pr:ref1} and using the map $ \Phi_{\rho} $, we identify the set $ \wu_{\rho} $ endowed with the symplectic form $ \omega $ and the Riemannian metric $ g $ with the set $ \ww_{a_{\rho}} $ endowed with the standard symplectic form $ \omega_{0} $ and the Riemannian metric $ g_{\rho} $. Accordingly, we can think of $ \wc $ as a cone field on $ \R^{d} \times \R^{d} $ such that $ \wy_{\rho} \subset \wc(y) $ for every $ y \in \wu_{\rho} $. Also, we identify the first return map $ \wt_{\rho} $ on $ \wu_{\rho} $ induced by $ \wt $ with the map $ \Phi^{-1}_{\rho} \circ \wt_{\rho} \circ \Phi_{\rho}:\ww_{a_{\rho}} \to \ww_{a_{\rho}} $. Note that the pushforward of the restriction of $ \mu $ to $ \wu_{\rho} $ under $ \Phi_{\rho} $ is equal to the restriction of the Lebesgue measure on $ \R^{d} \times \R^{d} $ to $ \ww_{a_{\rho}} $. Let us denote by $ \leb $ the measure $ \Phi_{\rho*}\mu $. The map $ \Phi^{-1}_{\rho} \circ \wt_{\rho} \circ \Phi_{\rho} $ preserves $ \leb $. Unless otherwise specified, this is the measure involved in all the theoretical measure statements throughout this section.

\begin{definition}	
	\label{de:spacedrettimes}		
Let $ \wu_{\rho} $ be as in Proposition \ref{pr:ref1}. Suppose that $ y \in \wu_{\rho} \setminus \wrp_{n} $ and $ y' = \wt^{n}y \in \wu_{\rho} $ for some $ n \ge N $. Let 
\[
T(y) = \left\{N \le i \le n:\wt^{i}y \in \wu_{\rho} \right \}
\]
Set $ i_{0} = 0 $, and consider the largest set $ \{i_{0},\ldots,i_{r}\} \subset T(y) $ such that $ i_{j+1}-i_{j} \ge N $ for every $ 0 \le j \le r-1 $. The non-negative integer $ r=r(y) $ is called the \emph{maximal number of $ N $-spaced return times of the orbit of $ y $ in the time-interval $ [N,n] $}. 
\end{definition}			
			
Let $ b_{\rho} = \sqrt{1-\rho^{4}} $. 
			
\begin{proposition}			                      
	\label{pr:ref2}
	Suppose that $ y \in \wu_{\rho} \setminus \wrp_{n} $ and $ y' = \wt^{n}y \in \wu_{\rho} $ for some $ n \ge N $. Then,
		\begin{enumerate}
			\item $ D_{y} \wt^{n} (\wy_{1/\rho}) \subset \wy_{\rho} $,
			\item $ D_{y'} \wt^{-n} (\wy^{'}_{\rho}) \subset \wy^{'}_{1/\rho} $,
			\item $ u \in \wy_{\rho} \Longrightarrow |\pi_{1}(\dy \wt^{n}u)| \ge b_{\rho} \cdot \rho^{-r(y)} |\pi_{1}(u)| $,
			\item $ u \in \wy^{'}_{1/\rho} \Longrightarrow |\pi_{2}(\dy \wt^{n} u)| \le  b^{-1}_{\rho} \cdot \rho^{r(y)} |\pi_{2}(u)| $.
		\end{enumerate}
\end{proposition}       

\begin{remark}
	To prove Theorem \ref{th:sinai}, we have to assume that $ 0<\rho<1/3 $ (see the end of Section 12 of \cite{liwo95} for more details). This is possible because $ 0<\w{\rho}<1/3 $.
\end{remark}
              
\subsection{Size of the stable and unstable manifolds}        
\label{su:localm}

The next step of the proof of the LET consists in showing that by shrinking the neighborhood $ \, \wu_{\rho} $ if necessary, the unstable(stable) manifold of almost every point of $ \, \wu_{\rho} $ has a uniform `size' or is `cut' by the sets $ \wt^{j} \wsn_{1}(\wt^{-j} \wsp_{1}) $ with $ j \ge N $. The results given below are formulated only for the unstable manifolds. Analogous results hold for stable manifolds as well.

Let $ \wu_{\rho} $ and $ a_{\rho} $ be as in Proposition \ref{pr:ref1}. Let $ \Lambda $ be the set associated to the sufficient point $ x $ as in Proposition \ref{pr:man1}. For every $ y \in \wu_{\rho} $ and $ \ep>0 $, let $ B(y_{i},\ep) = \{z \in \R^{d}: |z-\pi_{i}(y)| < \ep \} $ be the ball of $ \R^{d} $ centered at $ \pi_{i}(y) $ and of radius $ \ep $ for each $ i=1,2 $. 

\begin{definition}
	Let $ y \in \Lambda \cap \wu_{\rho} $. We say that the unstable manifold $ W^{u}_{y} $ has \emph{size} $ \ep>0 $ if there exists a $ C^{2} $ map $ \psi:B(\pi_{1}(y),\ep) \to \wv_{a_{\rho}} $ such that  the graph of $ \psi $ is contained in $ V^{u}_{y} $, and $ \ep $ is the largest number with this property. 
\end{definition}

The next lemma states that if the unstable manifold of a point of $ \wu_{\rho} $ has size $ \ep $, then the projection of that unstable manifold along the second component of $ \R^{d} \times \R^{d} $ is contained in a ball of radius $ \rho \ep $. As a consequence, we see that the stable spaces of the unstable manifold are contained inside the cone $ \wy_{\rho} $.

\begin{lemma}
	\label{le:verticalsize}
	Let $ y \in \Lambda \cap \wu_{\rho} $. If the unstable manifold $ W^{u}_{y} $ has size $ \ep $, then 
	\[ 
	\pi_{2}(W^{u}_{y}) \subset B(\pi_{2}(y),\rho \ep).
	\]
\end{lemma}
	
\begin{definition}
	Let $ y \in \Lambda \cap \wu_{\rho} $. We say that the unstable manifold $ W^{u}_{y} $ is \emph{cut} by a set $ \wa \subset \wm $ if the intersection of $ \D W^{u}_{y} $ and $ \wa $ is not empty.
\end{definition}
	
Define 
\[
\wu^{1}_{\rho,\eta} = \ww_{a_{\rho} - \eta/b_{\rho}} \subset \wu_{\rho} \qquad \text{for } 0 < \eta < a_{\eta} b_{\eta}.
\]

\begin{theorem} 
	\label{th:man2}
	If $ y \in \Lambda \cap \wu^{1}_{\rho,\eta} $ and the size of $ W^{u}_{y} $ is less than $ \eta $, then $ W^{u}_{y} $ is cut by the set $ \bigcup_{j \ge N} \wt^{j} \wsn_{1} $.	
\end{theorem}     

The proofs of Lemma \ref{le:verticalsize} and Theorem \ref{th:man2} are exactly as the ones of Lemma 9.6 and Theorem 9.7 of \cite{liwo95}. We stress that even if the existence of an homeomorphic extension of the map $ \wt $ up to the boundary of each connected component of $ \sman \setminus \wrp_{1}(\sman \setminus \wrn_{1}) $ is assumed in \cite{liwo95}, this property is not used in the proofs of Lemma 9.6 and Theorem 9.7 (cf. Remark \ref{re:notneeded}).

\subsection{Rectangles and coverings}
From now on, we assume that $ 0 < \rho < 1/3 $. This condition is required for proving Theorem \ref{th:sinai}. We also assume that $ 0 < \eta < a_{\rho} b_{\rho} $ is so small (i.e., $ \wu^{1}_{\rho,\eta} $ is so small compared to $ \wu_{\rho} $) that for every point $ y \in \Lambda \cap \wu^{1}_{\rho,\eta} $, the unstable manifold $ W^{u}_{y} $ can intersects the boundary of $ \wu_{\rho} $ only along its `vertical' part $ \D \nu_{a_{\rho}} \times \nu_{a_{\rho}} $ (cf. Lemma \ref{le:verticalsize}). Since $ \rho $ and $ \eta $ are fixed, to simplify notation, we set 
\[ 
\wu = \wu_{\rho} \qquad \text{and} \qquad \wu^{1} = \wu^{1}_{\rho,\eta}.
\]

We now recall a series of definitions from Sections 9-11 of \cite{liwo95}, which are needed to formulate and prove Theorem \ref{th:sinai}. 

\begin{definition}
	A \emph{rectangle} $ R(y,\zeta) \subset \R^{d} \times \R^{d} $ with the center at $ y \in \R^{d} \times \R^{d} $ and of size $ \zeta>0 $ is the Cartesian product of the closure of the balls $ B(\pi_{1}(y),\zeta/2) $ and $ B(\pi_{2}(y),\zeta/2) $, i.e., 
	\[
	R(y,\zeta) = \bar{B}(\pi_{1}(y),\zeta/2) \times \bar{B}(\pi_{2}(y),\zeta/2).
	\] 
\end{definition}	
	
\begin{definition}
	Let $ R(y,\zeta) $ be a rectangle, and suppose that $ z \in \Lambda \cap R(y,\zeta) $. We say that the unstable manifold $ W^{u}_{z} $ is \emph{connecting in} $ R(y,\zeta) $ if the intersection of $ W^{u}_{z} $ and $ R(y,\zeta) $ is the graph of a $ C^{2} $ map from the closed ball $ \bar{B}(\pi_{1}(y),\zeta/2) $ to the open ball $ B(\pi_{2}(y),\zeta/2) $. 
\end{definition}
	
\begin{definition}	
	The \emph{unstable core} of a rectangle $ R(y,\zeta) $ is given by the set 
	\[ 
	\left\{ z \in R(y,\zeta): \rho |\pi_{1}(z)-\pi_{1}(y)|+|\pi_{2}(z)-\pi_{2}(y)|<\frac{1}{2}(1-\rho)\zeta \right\}. 
	\] 
\end{definition}

The notions of a stable manifold connecting in a rectangle and the stable core of a rectangle can be defined similarly. The importance of the notions of the unstable and stable cores of a rectangle is due to the fact that if an unstable(stable) manifold intersects the unstable(stable) core of a rectangle $ R $, then the manifold is connecting in $ R $. The precise result is the following lemma, whose proof is identical to the one of Lemma 9.12 of \cite{liwo95}.

\begin{lemma}
Let $ R(y,\zeta) $ be a rectangle, and suppose that $ z \in \Lambda \cap R(y,\zeta) $. If the unstable manifold $ W^{u}_{z} $ intersects the unstable core of $ R(y,\zeta) $ and the size of $ W^{u}_{z} $ is greater than $ |\pi_{1}(z)-\pi_{1}(y)| + \zeta/2 $, then $ W^{u}_{z} $ is connecting in $ R(y,\zeta) $. 
\end{lemma}

Next, we introduce a family of neighborhoods of $ x $ approximating $ \wu^{1} $ from the inside. Define 
\[
\wu^{2}_{\delta} = \ww_{a_{\rho}-b_{\rho}/\eta-\delta/2} \qquad \text{for } 0 < \delta < 2(a_{\rho}-\eta/b_{\rho}).
\] 
Note that $ \wu^{2}_{\delta} \subset \wu^{1} $, and that $ \wu^{2}_{\delta} \to \wu^{1} $ as $ \delta \to 0^{+} $. 

\begin{definition}	
For every $ \delta>0 $ and $ c>0 $, let 
\[
\mathcal{N}_{\delta,c} = \left\{c \delta \cdot k \in \wu^{2}_{\delta}: k = (k_{1},k_{2}) \in \Z^{d} \times \Z^{d} \right\},
\]
and let 
\[ 
\mathcal{G}_{\delta,c} = \left\{R(y,\delta): y \in \mathcal{N}_{\delta,c}\right\}.
\] 
\end{definition}

The set $ \mathcal{N}_{\delta,c}$ is a lattice contained in $ \wu^{2}_{\delta} $, whereas the set $ \mathcal{G}_{\delta,c} $ is a collection of rectangles with center at the points of $ \mathcal{N}_{\delta,c}$ and of size $ \delta $. If $ c $ is sufficiently small, then $ \mathcal{G}_{\delta,c} $ is a covering of $ \, \wu^{2}_{\delta} $.
			
We recall that the measure used in the next definition is $ \leb $.			
			
\begin{definition}	
	\label{de:connectingrectangle}
	Let $ 0 < \alpha < 1 $. We say that a rectangle $ R \in \mathcal{G}_{\delta,c} $ is $ \alpha $\emph{-connecting in the unstable(stable) direction} if the measure of the intersection of the union of all unstable(stable) manifolds connecting in $ R $ with the unstable(stable) core of $ R $ is at least an $ \alpha $ fraction of the total measure of the unstable(stable) core. We say that $ R \in \mathcal{G}_{\delta,c} $ is \emph{$ \alpha $-connecting} if $ R $ is $ \alpha $-connecting in both the unstable and stable directions. 
\end{definition}

\subsection{Sketch of the `Proof of the Main Theorem' of \cite{liwo95}}
\label{su:maintheorem}
We now come to the central argument of this proof: it is the exact same argument of the `Proof of the Main Theorem' In Section 11 of \cite{liwo95}. This argument relies on two sets of results that still hold for our setting. The first set of results consists of Lemma 10.1 and Proposition 10.2 of Section 10 of \cite{liwo95}, which are valid for general non-uniformly hyperbolic systems with singularities. The proof of Proposition 10.2 makes use of the Hopf argument and the absolute continuity of the stable and unstable families. We recall that the Hopf argument was devised by Hopf to prove the ergodicity of the geodesic flow on a surface of negative curvature \cite{h} (see also Section 11 of \cite{liwo95} for a detailed account of this argument). The second set of results consists of Lemma 11.3 and Proposition 11.4 of Section 11 of \cite{liwo95}, which are abstract results (measure theory and combinatorics), and Theorem \ref{th:sinai} of this paper. We adopt the terminology of \cite{liwo95}, and so call Theorem \ref{th:sinai} Sinai's Theorem. For reasons of space, we do not repeat the `Proof of the Main Theorem' of \cite[Section 11]{liwo95} here, and instead refer the reader to the paper of Liverani and Wojtkowski. However, for completeness, we provide a sketchy description of this proof. The goal is to show that the neighborhood $ \wu^{1} $ is contained (mod 0) in one ergodic component of $ \wt $. Here are the main steps of the proof.

\begin{enumerate}
	\item Denote by $ \wk_{\delta} $ the subset of $ \mathcal{N}_{\delta,c} $ formed by points $ y \in \mathcal{N}_{\delta,c} $ for which the rectangle $ R(y,\delta) $ is $ \alpha $-connecting.
	\item Let $ g\wk_{\delta} $ be the largest subset of $ \wk_{\delta} $ with the property that for any two pints $ y,z \in g\wk_{\delta} $, there exist finitely many points  $ y_{0} = y,y_{1},\ldots,y_{n}=z $ belonging to $ g\wk_{\delta} $ such that $ y_{i} $ and $ y_{i+1} $ are nearest neighbors in $ g\wk_{\delta} $ for $ i=0,\ldots,n-1 $. If there are several such largest sets, then pick one of them.
	\item Let $ Y_{\delta} $ be the union of all the rectangles with center at $ g \wk_{\delta} $. The set $ Y_{\delta} $ belongs (mod 0) to one ergodic component of $ \wt $. To prove this claim, we need use the Hopf argument and the absolute continuity of the stable and unstable families (see Proposition 10.2 of \cite{liwo95}).
	\item Given two positive measure subsets $ A_{1} $ and $ A_{2} $ of the neighborhood $ \wu^{1} $, there exists $ \delta > 0 $ such that the intersections $ A_{1} \cap Y_{\delta} $ and $ A_{2} \cap Y_{\delta} $ have positive measure. To prove this claim, we need Sinai's theorem, Lemma 11.3 and Proposition 11.4 of \cite{liwo95}.
	\item The previous step allows us to conclude that if $ E_{1} $ and $ E_{2} $ are arbitrary ergodic components of $ \wt $ having intersection with $ \wu^{1} $ of positive measure, then $ E_{1} = E_{2} $ (mod 0). It follows that $ \wu^{1} $ is contained (mod 0) in one ergodic component of $ \wt $.
\end{enumerate}

\subsection{Sinai's Theorem}      
\label{su:sinai1}

In this and the next subsection, we formulate and prove Sinai's Theorem.

\begin{definition}
Let $ \wh $ be the set of all the rectangles of $ \wg $ that are not $ \alpha $-connecting, i.e.,  
\[
\wh = \left\{R \in \mathcal{G}_{\delta,c}: R \text{ is not } \alpha\text{-connecting}\right\}.
\]
Also, let $ \wh^{u(s)} $ be the set of all the rectangles of $ \wg $ that are not $ \alpha $-connecting in the unstable(stable) direction. 
\end{definition}

Of course, we have
\[ 
\wh = \wh^{s} \cup \wh^{u}.
\]

We can now formulate Sinai's Theorem.

\begin{theorem}[Sinai's Theorem] 
	\label{th:sinai}
	There exists $ 0<\alpha<1 $ such that 
	\[ 
	\lim_{\delta \to 0^{+}} \delta^{-1} \leb \left( \bigcup_{R \in \wh} R \right) = 0 \qquad \text{for all } c>0.
	\]
\end{theorem}

As in \cite{liwo95}, we prove only the unstable version of Theorem \ref{th:sinai} (i.e., with $ \wh $ replaced by $ \wh^{u} $). The stable version can be proved similarly.
	
We proceed as follows. We partition the set of all the rectangles that are not $ \alpha $-connecting in the unstable direction into two complementary sets. To describe precisely these sets, we introduce the concept of an $ (M,\alpha) $-nonconnecting rectangle.

\begin{definition}
	Let $ 0 < \alpha < 1 $. We say that a rectangle $ R \in \wh^{u} $ is \emph{$ (M,\alpha) $-nonconnecting} if at least $ 1-\alpha $ part of the measure of the unstable core of $ R $ consists of points whose unstable manifold is not connecting in $ R $ and is cut by the set $ \bigcup^{M}_{i=N} \wt^{i} \wsn_{1} $. 
\end{definition}

\begin{definition}
Let $ \wj $ be the set of all the rectangles of $ \wh^{u} $ that are $ (M,\alpha) $-nonconnecting. 
\end{definition}

Sinai's Theorem is a consequence of Propositions \ref{pr:sinai1} and \ref{pr:tb}. The first proposition concerns $ (M,\alpha) $-rectangles, whereas the second one concerns the rectangles that are not $ \alpha $-connecting (in the unstable direction) and are not $ (M,\alpha) $-connecting.

\begin{proposition} 
	\label{pr:sinai1}
	There exists $ 0<\alpha<1 $ such that
	\[ 
	\lim_{\delta \to 0^{+}} \delta^{-1} \leb \left( \bigcup_{R \in \wj} R \right) = 0 \qquad \text{for all } c>0 \text{ and } M \ge N.
	\]
\end{proposition}

The proof is exactly as the one of Proposition 12.2 of \cite{liwo95}. We stress that of Conditions L1-L4 only L1 and L2 are used in this proof. In particular, Condition L1 is required to be able to apply Proposition \ref{pr:vol}.

\begin{proposition}[Tail Bound]
	\label{pr:tb}
	For every $ \ep>0 $, there exist $ M_{\ep}>0 $ and $ \delta_{\ep}>0 $ such that
	\[
	\leb \left( \bigcup_{R \in \wh^{u} \setminus \mathcal{H}^{u}_{\delta,c,\alpha,M_{\ep}}} R \right) \le \ep \cdot \delta \qquad \text{for all } 0 < \delta < \delta_{\ep}, 0<\alpha<1 \text{ and } c>0.
	\]
\end{proposition}

The proof of this proposition is postponed to Subsection \ref{su:tb}. We now prove Sinai's Theorem.

\begin{proof}[Proof of Theorem \ref{th:sinai}]
		Recall that 
		\[ 
		\wh^{u} = \wj \cup (\wh^{u} \setminus \wj) \qquad \text{for } M \ge N.
		\] 
		Fix $ \ep>0 $. By Proposition \ref{pr:tb}, there exist $ M_{\ep} \ge N $ and $ \delta_{\ep} > 0 $ such that
		\begin{equation}
			\label{eq:step1}
		\leb \left(\bigcup_{R \in \wh^{u} \setminus \mathcal{H}^{u}_{\delta,c,\alpha,M_{\ep}}} R \right) \le \frac{\ep}{2} \delta \qquad \text{for } 0 < \delta < \delta_{\ep}.
		\end{equation}
		Now, apply Proposition \ref{pr:sinai1} to $ \wj $ with $ M = M_{\ep} $. It follows that there exists $ 0 < \alpha < 1 $ such that
		\begin{equation}
			\label{eq:step2}
			\leb \left(\bigcup_{R \in \mathcal{H}^{u}_{\delta,c,\alpha,M_{\ep}}} R \right) \le \frac{\ep}{2} \delta \qquad \text{for } c>0 \text{ and } \delta>0 \text{ sufficiently small}. 
		\end{equation}
		Inequalities \ref{eq:step1} and \ref{eq:step2} imply that there exists $ 0 < \alpha < 1 $ such that
		\[
		\leb \left(\bigcup_{R \in \wh^{u}} R \right) \le \ep \delta \qquad \text{for every } c>0 \text{ and } \delta>0 \text{ sufficiently small}.
		\]
\end{proof}

\subsection{Tail Bound} 
\label{su:tb}
To complete the proof of Sinai's Theorem, we need to prove Proposition \ref{pr:tb}. This proposition is the analog of the result proved in Section 13 of \cite{liwo95}, where it is called the `Tail Bound' estimate. The proof given in Section 13 of \cite{liwo95} is not valid in our setting, because it relies on the assumption that the cone field $ \wc $ is continuous on the entire set $ \sman \setminus \D \sman $, whereas we only assume $ \wc $ to be defined in the neighborhood $ U \cup \wt^{-N} U $. Proposition \ref{pr:tb} is proved at the end of this section, and is a straightforward consequence of Proposition \ref{pr:pretb}.
	
\begin{definition}
 	Given $ M>0 $, let $ Y_{\delta,M} $ the set of points of $ \Lambda \cap \mau^{1} $ whose unstable manifold has size smaller than $ \delta $ and is cut by the set $ \bigcup_{i \ge M+1} \wt^{i} \wsn_{1} $. 
\end{definition}

\begin{proposition}
	\label{pr:pretb} 
	For every $ \ep> 0 $, there exist $ M_{\ep} \ge N $ and $ \delta_{\ep}>0 $ such that  
	\[ 
	\leb (Y_{\delta,M_{\ep}}) \le \ep \cdot \delta \qquad \text{for every } 0 < \delta < \delta_{\ep}.
	\] 
\end{proposition}

To prove Proposition \ref{pr:pretb}, we need Lemma \ref{le:tb2} below. We observe that although Proposition \ref{pr:pretb} (as well as Proposition \ref{pr:tb}) concerns a subset of $ \wu^{1} $, the proof of the proposition uses an argument that is global, i.e., that is not restricted to $ \wu^{1} $ only. This fact can be clearly seen in the statement of Lemma \ref{le:tb2} which is about the existence of a cone field jointly invariant with $ \wc $ on a neighborhood of the singular set $ \wsn_{1} $. We also wish to stress that only Conditions L1, L3 and L4, but not L2, are used in the proof of Proposition \ref{pr:pretb}. Condition L3 is required to prove Lemma \ref{le:tb2}.

Let $ \wq $ be the quadratic form associated to the cone field $ \wc $ on $ U \cup \wt^{-N}U $ (see Subsection \ref{su:quad}). Recall that $ A(\ep) $ is the $ \ep $-neighborhood of the set $ A \subset \sman $. 

\begin{lemma}
	\label{le:tb2} 
	For every $ t>0 $ and $ 0<h<1 $, there exist an integer $ M_{t,h}>0 $, two compact subsets $ \ws_{t,h} $ and $ \we_{t,h} $ of $ \wsn_{1} $, a real number $ r_{t,h}>0 $ and a cone field $ \wp_{t,h} $ on $ \ws_{t,h}(r_{t,h}) $ such that 
	\begin{itemize}
		\item $ \wsn_{1} = \ws_{t,h} \cup \we_{t,h} $ and $ \Leb_{-}(\we_{t,h})<h $,
		\item the cone fields $ \wp_{t,h} $ and $ \wc $ are jointly invariant,
		\item if $ z \in \ws_{t,h}(r_{t,h}) $ and $ \wt^{j} z \in U $ with $ j \ge M_{t,h} $, then 
		\[ 
		\inf_{u \in \innt \wp_{t,h}(z) \setminus \{0\}} \frac{\sqrt{\wq_{\wt^{j}z}(D_{z} \wt^{j}u)}}{\|u\|}>t.
		\]
	\end{itemize}
\end{lemma}
\begin{proof}
	In this proof, the set $ B(y,r) $ denotes the open ball of $ \wm $ centered at $ y \in \wm $ and of radius $ r>0 $. 
	
	It is enough to prove the lemma with $ \wsn_{1} $ replaced by an arbitrary component $ \Sigma $ of $ \wsn_{1} $. Let us start by observing that Proposition \ref{pr:vol} applied to $ \Sigma $ implies that $ \Leb_{-}(\D \Sigma)=0 $. From this fact and Condition L3, we then see that $ \Leb_{-} $-a.e. point of the interior of $ \Sigma $ is u-essential. Now, let $ h>0 $. The regularity of the measure $ \Leb_{-} $ allows us to find a compact subset $ \Sigma_{1} $ of the interior of $ \Sigma $ such that $ \Leb_{-}(\Sigma \setminus \Sigma_{1})<h $ and every point of $ \Sigma_{1} $ is u-essential. Next, let $ t>0 $. By the definition of u-essential point, for every $ y \in \Sigma_{1} $, we can find a real $ p_{y}>0 $, an integer $ k_{y}>0 $ and a continuous invariant cone field $ \wp_{y} $ on $ B(y,p) \cup \wt^{k} B(y,p) $ that is jointly invariant with $ \wc $ such that 
	\begin{equation}
		\label{eq:expansion}
	\sigma^{*}_{\wp_{y}} (D_{z} \wt^{k}) > t \qquad \text{for every } z \in B(y,p). 
	\end{equation}

Let $ \wq_{\wc} $ and $ \wq_{\wp_{y}} $ denote the quadratic forms associated to cone fields $ \wc $ and $ \wp_{y} $, respectively. Since $ \wp_{y} $ and $ \wc $ are jointly invariant, it follows  that  
	\begin{equation}
		\label{eq:noncontraction}
	\inf_{u \in \innt \wp_{y}(z) \setminus \{0\}} \sqrt{\frac{\wq_{\wc}(D_{z} \wt^{j}u)}{\wq_{\wp_{y}}(u)}} \ge 1
	\end{equation}
	provided that $ z \in B(y,p) \cup \wt^{k} B(y,p) $ and $ \wt^{j}z \in U $ with $ j \ge 0 $ (see Remark \ref{re:invariance}). Combining \eqref{eq:expansion} and \eqref{eq:noncontraction}, we can conclude that if $ z \in B(y,p) $ and $ \wt^{j}z \in U $ with $ j \ge k $, then 
	\begin{equation}
		\label{eq:exp}
	\begin{split}
		\inf_{u \in \innt \wp_{y}(z) \setminus \{0\}} \frac{\sqrt{\wq_{\wc}(D_{z} \wt^{j}u)}}{\|u\|}
		& = \inf_{u \in \innt \wp_{y}(z) \setminus \{0\}} \frac{\sqrt{\wq_{\wp_{y}}(D_{z} \wt^{k}u)}}{\|u\|} \cdot \sqrt{\frac{\wq_{\wc}(D_{z} \wt^{j}u)}{
		\wq_{\wp_{y}}(D_{z}\wt^{k}u)}} \\
		& \ge \sigma^{*}_{\wp_{y}} (D_{z} \wt^{k}) \cdot \inf_{u \in \innt \wp_{y}(\wt^{k}z) \setminus \{0\}} \sqrt{\frac{\wq_{\wc}(D_{\wt^{k} z} \wt^{j-k}u)}{\wq_{\wp_{y}}(u)}} \\
		& > t.
	\end{split}
	\end{equation}

Because of the compactness of $ \Sigma_{1} $, there exist $ m>0 $ points $ y_{1},\ldots,y_{m} $ of $ \Sigma_{1} $ such that $ \Sigma_{1} \subset \bigcup^{m}_{i=1} B(y_{i},p(y_{i})) $. Let $ B_{1} = B(y_{1},p(y_{1})) $, and let 
\[ 
B_{i} = B(y_{i},p(y_{i})) \setminus \bigcup^{i-1}_{j=1} B_{j} \qquad \text{for } i = 2,\ldots,m. 
\]
The sets $ B_{1},\ldots,B_{m} $ are pairwise disjoint and $ \bigcup^{m}_{i=1} B_{i} =  \bigcup^{m}_{i=1} B(y_{i},p(y_{i})) $. Next, we define the cone field $ \wp $ on $ \bigcup^{m}_{i=1} B_{i} $ by setting 
\[
\wp = \wp_{y_{i}} \quad \text{on} \quad B_{i} \qquad \text{for each } i=1,\ldots,m.
\] 
It is easy to see that $ \wp $ and $ \wc $ are jointly invariant. Now, denote by $ M $ the maximum of $ k(y_{1}),\ldots,k(y_{m}) $. Then, inequality \eqref{eq:exp} implies that if $ z \in \bigcup^{m}_{i=1} B_{i} $ and $ \wt^{j}z \in U $ with $ j \ge M $, then 
\begin{equation}
	\label{eq:four}
	\inf_{u \in \innt \wp(z) \setminus \{0\}} \frac{\left(\wq_{\wc}(D_{z} \wt^{j}u)\right)^{\frac{1}{2}}}{\|u\|} > t.
\end{equation}

Since the compact set $ \Sigma_{1} $ is contained in the open set $ \bigcup^{m}_{i=1} B_{i} $, we can find a compact set $ K \subset \wm $ and two numbers $ q,r>0 $ such that 
\[
\Sigma_{1} \subset \Sigma_{1}(q) \subset K \subset K(r) \subset \bigcup^{m}_{i=1} B_{i}.
\]
Let us define 
\[ 
\ws = K \cap \Sigma \qquad \text{and} \qquad \we = \Sigma \setminus \Sigma_{1}(q).
\] 
It is clear that the sets $ \ws $ and $ \we $ are compact and $ \ws \cup \we = \Sigma $. Moreover, we have 
\[
\Leb_{-}(\we) \le \Leb_{-}(\Sigma \setminus \Sigma_{1})<h.
\] 
The set $ \we $ is the analog of the set $ E \cup \overline{B^{\zeta}_{M}} $ in \cite[Section 13]{liwo95}. Finally, since  $ K \subset K(r) $, we see that $ \ws(r) \subset K(r) $, which in turn implies that inequality \eqref{eq:four} holds for every $ z \in \ws(r) $ such that $ \wt^{j}z \in U $ with $ j \ge M $. To complete the proof, we set $ M_{t,h} = M $, $ \ws_{t,h} = \ws $, $\we_{t,h} = \we,r_{t,h} = r $ and $ \wp_{t,h} = \wp $.
\end{proof}

\begin{definition}
	For every $ z \in Y_{\delta,M} $, define
	\[
	m(z) = \min \left\{i \ge M+1: \D W^{u}_{z} \cap \wt^{i} \wsn_{1} \neq \emptyset \right\},
	\]
	and 
	\[
	k(z) = \# \left\{i: 1 \le i \le m(z)-M \text{ and } \wt^{-i}z \in \mau^{1} \right\}.
	\] 
\end{definition}
Note that $ k(z) $ is just the number of returns of $ z $ to $ \wu^{1} $ in the time-interval $ [-1,m(z)-M] $. 

\begin{definition}
	\label{de:partition}
		For every $ m \ge M+1 $ and $ k \ge 0 $, let 
		\[ 
		Y^{k}_{m} = \left\{z \in Y_{\delta,M}: m(z)=m \text{ and } k(z)=k \right\}. 
		\] 
\end{definition}

It is easy to see that $ Y_{\delta,M} = \bigcup_{k \ge 0 } \bigcup_{m \ge M+1} Y^{k}_{m} $.

The proof of the next lemma is exactly as the proof of the analogous statement in \cite[Section 13, page 61]{liwo95}. We include this proof here for the convenience of the reader.

\begin{lemma} 
	\label{le:secp}  
	For every $ k \ge 0 $ and $ M>0 $, we have
	\[ 
	\leb \left(\bigcup_{m \ge M+1} Y^{k}_{m}\right) \le \mu \left(\bigcup_{m \ge M+1} \wt^{-m} Y^{k}_{m}\right).
	\]
\end{lemma}

\begin{proof}
	We claim that the sets $ \wt^{-m_{1}} Y^{k}_{m_{1}} $ and $ \wt^{-m_{2}} Y^{k}_{m_{2}} $ are disjoint provided that $ m_{1} \neq m_{2} $. On the contrary, suppose that this is not true. Then, we must have $ y \in \wt^{-m_{1}}Y^{k}_{m_{1}} \cap \wt^{-m_{2}}Y^{k}_{m_{2}} \neq \emptyset $ for some $ m_{1}<m_{2} $. Let $ z_{1} = \wt^{m_{1}}y $ and $ z_{2} = \wt^{m_{2}}y $. Since both $ z_{1} $ and $ z_{2} $ belong to $ \wu^{1} $, it follows that $ k(z_{2}) \ge k(z_{1})+1 $, which contradicts the fact that $ z_{1} \in Y^{k}_{m_{1}} $ and $ z_{2} \in Y^{k}_{m_{2}} $. By our claim, the invariance of $ \mu $ and the equality $ \leb\left(\bigcup_{m \ge M+1} Y^{k}_{m}\right) = \mu \left(\bigcup_{m \ge M+1} Y^{k}_{m}\right) $ (see the discussion after Proposition \ref{pr:ref1}), we can conclude that
		\begin{align*}
		\leb\left(\bigcup_{m \ge M+1} Y^{k}_{m}\right) & = \mu \left(\bigcup_{m \ge M+1} Y^{k}_{m}\right) \le \sum_{m \ge M+1} \mu(Y^{k}_{m}) \\
		& \le \sum_{m \ge M+1} \mu \left(\wt^{-m} Y^{k}_{m}\right) \le \mu \left(\bigcup_{m \ge M+1} \wt^{-m} Y^{k}_{m}\right).
		\end{align*}
\end{proof}

Before proving Proposition \ref{pr:pretb}, we need to make a few last observations. 
	
\subsubsection*{\bf Semi-norms and lengths} Since $ \wy_{\rho} \subset \innt \wc(y) $, it follows that $ \wq_{y}(u)>0 $ for every $ u \in \wy_{\rho} $ and $ y \in \mau $. Then, it is not difficult to see that the functional 
	\[ 
	u \in \wy_{\rho} \mapsto \sqrt{\wq_{y}(u)} 
	\] 
	defines a semi-norm on $ \wy_{\rho} $ for every $ y \in \wu $, which we denote by $ \|\cdot\|_{\wq} $. Another semi-norm $ \|\cdot\|_{1} $ on $ \wy_{\rho} $ for every $ y \in \mau $ is obtained by setting 
	\[ 
	\|u\|_{1} = |\pi_{1}(u)| \qquad \text{for every } u \in \wy_{\rho}. 
	\] 
	Recall that $ \|\cdot\| $ is the norm generated by the Riemannian metric $ g $. The semi-norms $ \|\cdot\|_{\wq} $, $\|\cdot\|_{1} $, $ \|\cdot\| $ vary continuously with $ y $ and are equivalent on $ \wu $. In particular, we have
	\begin{equation}
		\label{eq:equivalencenorm}
		\|\cdot\| \le (1+\rho^{2})^{\frac{1}{2}} \|\cdot\|_{1} \qquad \text{and} \qquad \|\cdot\|_{\wq} \le q \|\cdot\|_{1},
	\end{equation}
	where
	\[
	q = \sup_{u \in \wy_{\rho} \setminus \{0\}} \frac{\sqrt{\wq(u)}}{\|u\|_{1}}.
	\]

Now, given a differentiable curve $ \gamma $ with parameterization $ \alpha \mapsto \gamma(\alpha) \in \mau $, let us denote by $ \ell_{\wq}(\gamma) $, $ \ell_{1}(\gamma) $, $ \ell(\gamma) $ the length of $ \gamma $ with respect to the semi-norms $ \|\cdot\|_{\wq} $, $ \|\cdot\|_{1} $, $ \|\cdot\| $, respectively. Relations \eqref{eq:equivalencenorm} imply 
	\begin{equation}
		\label{eq:equivalencelength}
	\ell(\gamma) \le (1+\rho^{2})^{\frac{1}{2}} \ell_{1}(\gamma) \qquad \text{and} \qquad \ell_{\wq}(\gamma) \le q \ell_{1}(\gamma).
	\end{equation}

We are now in a position to prove Lemma \ref{pr:pretb}. 

\begin{proof}[Proof of Proposition \ref{pr:pretb}]
	Recall that $ \rho $ and $ \eta $ have been fixed, and do not depend on $ \delta $. Also, recall that $ \eta $ has been chosen so small that for every point $ y \in \Lambda \cap \wu^{1} $, the unstable manifold $ W^{u}_{y} $ can intersects the boundary of $ \wu $ only along its `vertical' part $ \D \nu_{a_{\rho}} \times \nu_{a_{\rho}} $ by Lemma \ref{le:verticalsize}. Now, let $ t>0 $ and $ 0 < h < 1 $. Later on, we will choose $ t $ and $ h $ properly. Let $ \we_{t,h} $, $ \ws_{t,h} $, $ M_{t,h} $, $ r_{t,h} $, $ \wp_{t,h} $ be the compact subsets of $ \, \wsn_{1} $, the positive numbers and the cone field, respectively, as in Lemma \ref{le:tb2}. Recall that $ \wc $ and $ \wp_{t,h} $ are jointly invariant. By Lemma \ref{le:tangent}, we know that there exists a set $ Z_{t,h} $ of zero measure such that  
	\begin{equation}
		\label{eq:insidecone}
		T_{z} W^{u}_{y} \subset \wp_{t,h}(z) \qquad \text{for } z \in W^{u}_{z} \cap \ws_{t,h}(r_{t,h}) \text{ and } y \in \Lambda \setminus Z_{t,h}.
	\end{equation}
	
	Now, let $ y \in Y^{k}_{m} \setminus \wt^{m}Z_{t,h} $. Since the intersection of $ \D W^{u}_{y} $ and $ \wt^{m} \wsn_{1} $ is not empty, there exist $ w \in \D W^{u}_{y} \cap \wt^{m} \wsn_{1} $ and a $ C^{1} $ curve $ \gamma:[0,1] \to W^{u}_{y} $ such that $ \gamma(0)=y $ and $ \gamma(1) = w $. We have $ \ell_{1}(\gamma) < \delta $, because the size of $ W^{u}_{y} $ is less than $ \delta $.
	
	The rest of this proof consists of two parts. In the first part, we estimate $ \ell(\wt^{-m}\gamma) $. Since $ \wt^{-m} w \in \wsn_{1} $, such an estimate allows us to conclude that the curve $ \wt^{-m}\gamma $ and in particular the point $ \wt^{-m}y $ are contained in a neighborhood of $ \wsn_{1} $ of radius equal to $ \ell(\wt^{-m} \gamma) $. This fact is used in combination with Lemma \ref{le:secp} in the second part of the proof to obtain the wanted estimate on $ \leb(Y_{\delta,M}) $. To find an estimate for $ \ell(\wt^{-m} \gamma) $, we use first the expanding properties of $ \wt^{n} $ along the directions contained in the cone $ \wy_{\rho} $ (see Proposition \ref{pr:ref2}), and then the expanding properties of the map $ \wt^{-m+n} $ deduced in Lemma \ref{le:tb2} and assumed in Condition L4. 
	
	Our first task is to estimate $ \ell(\wt^{-m} \gamma) $. Note that $ \wt^{-m} \gamma $ is indeed a $ C^{1} $ curve, because the restriction of $ \wt^{-m} $ to the unstable manifold $ W^{u}_{y} $ is a diffeomorphism. Let $ 0 < n \le m-M $ be the time when $ y $ makes its $ k $th return to $ \mau^{1} $. We claim that the curve $ \wt^{-n} \gamma $ is contained in $ \wu $ if $ 0 < \delta < \rho \eta $. On the contrary, suppose that $ \wt^{-n} \gamma $ is not contained in $ \wu $ for $ 0 < \delta < \rho \eta $. Since our choice of $ \eta $ guarantees that curve $ \wt^{-n}\gamma $ intersects the `vertical' boundary $ \D \nu_{a_{\rho}} \times \nu_{a_{\rho}} $ of $ \wu $, and $ \mau $ is the $ \eta b^{-1}_{\rho} $-neighborhood of $ \mau^{1} $ and $ \wt^{-n}y \in \wu^{1} $, it follows that the $ \ell_{1} $-length of the connected component of $ \wu \cap \wt^{-n} \gamma $ containing $ \wt^{-n}y $ is greater than $ \eta b^{-1}_{\rho} $. By Lemma \ref{le:verticalsize}, the tangent space of the unstable manifold $ W^{u}_{\wt^{-n}y} $ at the point $ z $ is contained in $ \wy_{\rho} $ for every $ z \in \wu \cap \wt^{-n} \gamma $. We can then apply Proposition \ref{pr:ref2} to $ \wu \cap \wt^{-n} \gamma $, and conclude that 
	\[
	\ell_{1}(\gamma) \ge b_{\rho} \rho^{-r(\wt^{-n}y)} \ell_{1}(\wu \cap \wt^{-n} \gamma),
	\]
	where $ r(\wt^{-n}y) $ is the maximal number of $ N $-spaced returns of $ \wt^{-n}y $ to $ \wu $ (see Definition \ref{de:spacedrettimes}). It is easy to see that $ r(\wt^{-n}y) \le k/N-1 $. Therefore 
	\begin{equation}
		\label{eq:inequalityl1}
		\ell_{1}(\gamma) \ge b_{\rho} \rho^{1-k/N} \ell_{1}(\wu \cap \wt^{-n} \gamma).
	\end{equation}
	But $ \rho<1 $ so that 
	\[ 
	\ell_{1}(\gamma) \ge \rho^{1-k/N} \eta \ge \rho \eta > \delta.
	\]
	which contradicts the fact that $ \ell_{1}(\gamma) < \delta $. 

Let
\[
\lambda = \rho^{1/N}, \qquad c_{1} = \frac{1}{b_{\rho} \rho}, \qquad c_{2} = \frac{1}{\rho(1-\rho^{2})^{1/2}}, \qquad c_{3} = c_{2} q.
\]
Since we have established that $ \wt^{-n} \gamma \subset \wu $, using again \eqref{eq:inequalityl1}, we obtain
\begin{equation}
	\label{eq:inequalityl2}
	\ell_{1}( \wt^{-n} \gamma) \le c_{1} \lambda^{k} \ell_{1}(\gamma) < c_{1} \lambda^{k} \delta.
\end{equation}
From \eqref{eq:equivalencelength}, it follows that
\[
\ell(\wt^{-n} \gamma) < c_{2} \lambda^{k} \delta \qquad \text{and} \qquad \ell_{\wq}(\wt^{-n} \gamma) < c_{3} \lambda^{k} \delta.
\]

Since $ \wt^{-m} w \in \wsn_{1} $ and $ \wsn_{1} = \we_{t,h} \cup \ws_{t,h} $, we have two possibilities: i) $ \wt^{-m} w \in \we_{t,h} $, and ii) $ \wt^{-m} w \in \ws_{t,h} $. To estimate $ \ell(\wt^{-m} \gamma) $, we consider the two cases separately.

Case i): suppose that $ \wt^{-m} w \in \we_{t,h} $. Let $ \beta $ and $ \xi $ be the numbers as in Condition L4. We claim that the curve $ \wt^{-m} \gamma $ is contained in $ \wsn_{1}(\xi) $ if $ \delta<\xi/(\beta\sqrt{1+\rho^{2}}) $. We argue again by contradiction. If our claim is not true, then since $ \wt^{-m} w \in \wsn_{1} $, it follows that 
\[
\ell(\wsn_{1}(\xi) \cap \wt^{-m} \gamma) \ge \xi.
\]
By Condition L4, we then obtain
\[
\ell(\gamma) \ge \xi/\beta > \delta (1+\rho^{2})^{\frac{1}{2}},
\]
which together with \eqref{eq:equivalencelength} implies
\[
\ell_{1}(\gamma) > \delta
\]
contradicting the fact that $ \ell_{1}(\gamma) < \delta $. Now, that we know that $ \wt^{-m} \gamma \subset \wsn_{1}(\xi) $, Condition L4 implies that the tangent spaces of the curve $ \wt^{-n} \gamma $ contracts uniformly by a factor $ \beta $ under the action of the differential $ D \wt^{n-m} $. Therefore,
\[
\ell(\wt^{-m} \gamma) \le \beta \ell(\wt^{-n} \gamma) < \beta c_{2} \lambda^{k} \delta.
\]
From this inequality, we see that the curve $ \wt^{-m} \gamma $ is contained in the neighborhood of $ \we_{t,h} $ of radius equal to $ \beta c_{2} \lambda^{k} \delta $. In particular,
\begin{equation}
	\label{eq:neighborhood1}
	\wt^{-m}y \in \we_{t,h}(\beta c_{2} \lambda^{k} \delta).
\end{equation}

Case ii): suppose that $ \wt^{-m} w \in \ws_{t,h} $. We claim that the curve $ \wt^{-m} \gamma $ is contained in $ \ws_{t,h}(r_{t,h}) $ for $ \delta<t r_{t,h}/q $. As before, we argue by contradiction. Suppose that our claim is not true. Then, since $ \wt^{-m} w \in \ws_{t,h} $, it follows that 
\begin{equation}
	\label{eq:equationr}
\ell(\wt^{-m} \gamma) \ge r_{t,h}.
\end{equation}
Since $ \wt^{-m}y \in \Lambda \setminus Z_{t,h} $ and of course $ \wt^{-m} \gamma \subset W^{u}_{\wt^{-m}y} $, Relation \eqref{eq:insidecone} implies 
	\[
	T_{z} \wt^{-m} \gamma \subset \wp_{t,h}(z) \qquad \text{for every } z \in \ws_{t,h}(r_{t,h}) \cap \wt^{-m} \gamma.
	\]
By applying Lemma \ref{le:tb2} to $ T_{z} \wt^{-m} \gamma $ for every $ z \in \ws_{t,h}(r_{t,h}) \cap \wt^{-m} \gamma $, we obtain
\[
\ell_{\wq}(\gamma) \ge t r_{t,h} > q \delta.
\]
This inequality implies
\[
\ell_{1}(\gamma) > \delta,
\]
which contradicts the fact that $ \ell_{1}(\gamma) < \delta $. Therefore $ \wt^{-m} \gamma \subset \ws_{t,h}(r_{t,h}) $, and so Lemma \ref{le:tb2} implies 
\[
\ell(\wt^{-m} \gamma) < \frac{1}{t}\ell_{\wq}(\wt^{-n} \gamma) \le \frac{1}{t} c_{3} \lambda^{k} \delta,
\]
which is the wanted estimate of $ \ell(\wt^{-m} \gamma) $. Since $ \wt^{-m} w \in \ws_{t,h} $, we see that $ \wt^{-m} \gamma $ is contained inside the neighborhood of $ \ws_{t,h} $ of radius $ c_{3}\lambda^{k}\delta/t $. In particular,
\begin{equation}
	\label{eq:neighborhood2}
	\wt^{-m}y \in \ws_{t,h} (t^{-1} c_{3} \lambda^{k} \delta).
\end{equation}

Now, suppose that $ \delta < \min\{\rho \eta,\xi/(\beta\sqrt{1+\rho^{2}},t r_{t,h}/q\} $. Combining inclusions \eqref{eq:neighborhood1} and \eqref{eq:neighborhood2} together, we obtain
\[
\wt^{-m} Y^{k}_{m} \setminus Z_{t,h} \subset \we_{t,h}(\beta c_{2} \lambda^{k} \delta) \cup \ws_{t,h} (t^{-1} c_{3} \lambda^{k} \delta).
\]
Taking the union of the sets $ \wt^{-m} Y^{k}_{m} $ for $ m \ge M_{t,h} + 1 $, it follows that
\[
\left(\bigcup_{m \ge M_{t,h}+1} \wt^{-m} Y^{k}_{m}\right) \setminus Z_{t,h} \subset \we_{t,h}(\beta c_{2} \lambda^{k} \delta) \cup \ws_{t,h} (t^{-1} c_{3} \lambda^{k} \delta).
\]
Recall that $ \mu(Z_{t,h}) = 0 $. Using $ \mu \le \tau \Leb $ and Proposition \ref{pr:vol}, we get
\begin{align*}
\mu \left(\bigcup_{m \ge M_{t,h}+1} \wt^{-m} Y^{k}_{m} \right) & \le \mu\left(\we_{t,h}(\beta c_{2} \lambda^{k} \delta)\right) + \mu\left(\ws_{t,h} (t^{-1} c_{3} \lambda^{k} \delta)\right) \\
& \le \tau \Leb\left(\we_{t,h}(\beta c_{2} \lambda^{k} \delta)\right) + \tau \Leb \left(\ws_{t,h} (t^{-1} c_{3} \lambda^{k} \delta)\right) \\
& \le 2 \tau \left(\beta c_{2} h + \frac{1}{t} c_{3} \Leb_{-}(\wsn_{1}) \right) \lambda^{k} \delta.
\end{align*}
By Lemma \ref{le:secp}, it follows that
\begin{align*}
\leb\left(\bigcup_{m \ge M_{t,h}+1} Y^{k}_{m}\right) & \le \mu \left(\bigcup_{m \ge M_{t,h}+1} \wt^{-m} Y^{k}_{m}\right) \\
& \le 3 \tau \left(\beta c_{2} h + \frac{1}{t} c_{3} \Leb_{-}(\wsn_{1}) \right) \lambda^{k} \delta
\end{align*}
so that 
\begin{equation}
	\label{eq:estimatey}
	\begin{split}
\leb(Y_{\delta,M_{t,h}}) & \le \leb\left(\bigcup_{k \ge 0} \bigcup_{m \ge M_{t,h}+1} Y^{k}_{m}\right) \\ 
& \le  3 \tau \left(\beta c_{2} h + \frac{1}{t} c_{3} \Leb_{-}(\wsn_{1}) \right) \delta \sum_{k \ge 0} \lambda^{k} \\ 
& \le 3 \tau \left(\beta c_{2} h + \frac{1}{t} c_{3} \Leb_{-}(\wsn_{1}) \right) \frac{\delta}{1-\lambda}.
	\end{split}
\end{equation}
We can now finish the proof. Fix $ \ep>0 $. There exist $ t>0 $ and $ 0 < h < 1 $ such that 
\[
 3 \tau \left(\beta c_{2} h + \frac{1}{t} c_{3} \Leb_{-}(\wsn_{1}) \right) < \ep.
\]
For such $ t $ and $ h $, inequality \eqref{eq:estimatey} implies that there exist 
\[ 
M_{\ep} = M_{t,h} \qquad \text{and} \qquad \delta_{\ep} = \min\left\{\rho \eta,\frac{\xi}{\beta\sqrt{1+\rho^{2}}},\frac{t r_{t,h}}{q}\right\} 
\] 
such that 
\[
\leb(Y_{\delta,M_{\ep}}) \le \ep \delta \qquad \text{for } 0 < \delta < \delta_{\ep}.
\]
The proof is complete.			
\end{proof}

We are now in a position to prove Proposition \ref{pr:tb}.

\begin{proof}[Proof of Proposition \ref{pr:tb}]
The proposition is an immediate consequence of Proposition \ref{pr:pretb}, once we have proved the following inequality
	\begin{equation}
		\label{eq:MYinequality}
		\leb \left(\bigcup_{R \in \wh^{u} \setminus \wj} R \right) \le \frac{k(c)}{1-\alpha} \cdot \frac{2}{1-2\rho} \cdot \leb \left(Y_{\delta,M}\right),
	\end{equation}
	where $ k(c) > 0 $ denotes the maximum number of rectangles of $ \mathcal{G}_{\delta,c} $ whose intersection is not empty. Now, the proof of the previous inequality goes as follows. From the definitions of $ Y_{\delta,M} $ and the unstable core of a rectangle, it follows that
	\[
	\left(\wh^{u} \setminus \wj\right) \le k(c) \cdot \frac{\mu(Y_{\delta,M})}{(1-\alpha) \mu(\text{Unstable core of } R)}.
	\]
	Therefore
	\[
	\leb \left(\wh^{u} \setminus \wj\right) \le \frac{k(c)}{1-\alpha} \cdot \frac{\mu(R)}{\mu(\text{Unstable core of } R)} \cdot \mu(Y_{\delta,M}).
	\]
	Finally, an easy computation gives 
	\[
	\frac{\mu(R)}{\mu({\text{Unstable core of }R})} = \frac{2}{1-2\rho} \qquad \text{for }  0 < \rho < 1/3. 
	\]
\end{proof}

\subsection{Conclusion of the proof} 
In Subsection \ref{su:maintheorem}, we have proved that the neighborhood $ \wu^{1} $ is contained (mod 0) in one ergodic component of $ \wt $. We can actually obtain a stronger conclusion, namely, the neighborhood $ \wu^{1} $ belongs to one ergodic component of $ \wt^{m} $ for every $ m>0 $. In fact, since $ \wt^{m} $ preserves the measure $ \mu $, and has the same stable and unstable manifolds of $ \wt $, the whole argument delineated in this section works not just for $ \wt $ but for $ \wt^{m} $ as well with $ m>0 $.
 
We now show that $ \wu^{1} $ is indeed contained (mod 0) in a Bernoulli component of $ \wt $. 

\begin{lemma}
	\label{le:bernoulli}
	The neighborhood $ \, \wu^{1} $ is contained (mod 0) in a Bernoulli component of $ \wt $.
\end{lemma}

\begin{proof}
Let $ E $ be the ergodic component containing (mod 0) the neighborhood $ \wu^{1} $. It is clear that $ E $ has positive measure and non-zero Lyapunov exponents almost everywhere. Thus, let $ B_{1},\ldots,B_{m} $ be the Bernoulli components of $ \wt $ whose union gives $ E $. From the definition of a Bernoulli component, we see that the sets $ B_{1},\ldots,B_{m} $ are ergodic components of $ \wt^{m} $. By the considerations at the beginning of this subsection, it follows that $ \wu^{1} $ belongs (mod 0) to an ergodic component of $ \wt^{m} $. We can therefore conclude that $ \wu^{1} $ must be contained (mod 0) in one of the sets $ B_{1},\ldots,B_{m} $.
\end{proof}

To complete the proof of Theorem \ref{th:LET}, we observe that since $ \wt^{l} $ is a local diffeomorphism at $ x $, there exists a neighborhood $ \mathcal{O} $ of $ x $ such that $ \wt^{l}|_{\mathcal{O}} $ is a diffeomorphism and $ \wt^{l} \mathcal{O} \subset \wu^{1} $. By Lemma \ref{le:bernoulli}, the set $ \mathcal{O} $ belongs (mod 0) to a Bernoulli component of $ \wt $.

\end{document}